\documentclass[11pt]{article}
\usepackage{amssymb,amsthm,amsmath,float,cite,enumerate}
\usepackage[linesnumbered,ruled,lined]{algorithm2e}
\usepackage{geometry}
\usepackage{tikz}

\newtheorem{theorem}{Theorem}[section]
\newtheorem{lemma}[theorem]{Lemma}
\newtheorem{corollary}[theorem]{Corollary}

\newtheorem{claim}{Claim}[section]
\newtheorem*{claim*}{Claim}

\usepackage{setspace}
\usepackage{lineno}

\begin{document}

\onehalfspace

\title{On some tractable and hard instances for\\ partial incentives and target set selection}

\author{ Stefan Ehard \and Dieter Rautenbach}

\date{}

\maketitle

\begin{center}
Institut f\"{u}r Optimierung und Operations Research, 
Universit\"{a}t Ulm, Ulm, Germany,
\{\texttt{stefan.ehard,dieter.rautenbach}\}\texttt{@uni-ulm.de}\\[3mm]
\end{center}

\begin{abstract}
A widely studied model for influence diffusion in social networks are {\it target sets}.
For a graph $G$ and an integer-valued threshold function $\tau$ on its
vertex set, a {\it target set} or {\it dynamic monopoly} is a set of vertices of $G$ such that
iteratively adding to it vertices $u$ of $G$ that have at least
$\tau(u)$ neighbors in it eventually yields the entire vertex set of $G$. 
This notion is limited to the binary choice 
of including a vertex in the target set or not,
and Cordasco et al.~proposed {\it partial incentives}
as a variant allowing for intermediate choices.

We show that finding optimal partial incentives is 
hard for chordal graphs and planar graphs
but tractable 
for graphs of bounded treewidth and 
for interval graphs with bounded thresholds.
We also contribute some new results about target set seletion on planar graphs
by showing the hardness of this problem, and by
describing an efficient $O(\sqrt{n})$-approximation algorithm 
as well as a PTAS for the dual problem of finding a maximum degenerate set.
\end{abstract}

{\small 
\begin{tabular}{lp{13cm}}
{\bf Keywords:} & Dynamic monopoly; target set; partial incentive; degenerate set; 
chordal graph;
planar graph;
treewidth;
interval graph
\end{tabular}
}

\section{Introduction}
Target sets, also known as dynamic monopolies,
are a popular model for social influence and the spread of opinions in social networks 
\cite{drro,keklta}.
For unrestricted instances, optimal target sets are hard to find 
\cite{cedoperasz,ch},
and efficient algorithms are only known for quite restricted instances 
such as tree-structured graphs 
\cite{behelone,beehpera,kylivy}.
Constructing a target set involves binary decisions for every vertex of the underlying graph,
and Cordasco et al.~\cite{cogareva} recently proposed 
so-called partial incentives as a more subtle way of influencing a network.
Next to 
hardness results and bounds 
similar to those known for target sets \cite{acbewo,ch,re},
they describe an efficient algorithm
computing optimal partial incentives 
for trees and complete graphs.
In the present paper we present several new results concerning 
the complexity as well as tractable cases of partial incentives and (the dual of) target sets in planar graphs. 

Before we can explain our contribution in detail,
we need some terminology and notation.

We only consider finite, simple, and undirected graphs.
A {\it threshold function} for a graph $G$ is a function 
from its vertex set $V(G)$ to the set of integers.
Let $\tau\in \mathbb{Z}^{V(G)}$ be a threshold function for $G$.
For a set $D$ of vertices of $G$,
the {\it hull $H_{(G,\tau)}(D)$ of $D$ in $(G,\tau)$}
is the smallest set $H$ of vertices of $G$ such that 
$D\subseteq H$, and 
$u\in H$ for every vertex $u$ of $G$ with $|H\cap N_G(u)|\geq \tau (u)$.
The set $H_{(G,\tau)}(D)$ is obtained 
by starting with $D$, and 
iteratively adding vertices $u$ 
that have at least $\tau(u)$ neighbors in the current set
as long as possible.
A set $D$ of vertices of $G$ is a {\it dynamic monopoly} or a {\it target set} of
$(G,\tau)$ if $H_{(G,\tau)}(D)=V(G)$. The minimum order
of a dynamic monopoly of $(G,\tau)$ is denoted by ${\rm dyn}(G,\tau)$.

It has been observed several times~\cite{ch,acbewo,re} 
that the dual of the notion of a dynamic monopoly
is the notion of a degenerate set. 
For a function $\kappa\in\mathbb{Z}^{V(G)}$, 
a set $I$ of vertices of $G$ is {\it $\kappa$-degenerate} in $G$
if there is a linear order $u_1,\ldots,u_k$ of the vertices in $I$
such that $u_i$ has at most $\kappa(u_i)$ neighbors in $\{ u_j:j\in [i-1]\}$ for every $i\in [k]$,
where $[n]$ denotes the set of positive integers at most $n$ 
for every integer $n$, and $[n]_0=[n]\cup\{0\}$.
The maximum cardinality of a $\kappa$-degenerate set of $G$ is denoted by $\alpha(G,\kappa)$.
If $\kappa,\tau\in\mathbb{Z}^{V(G)}$ are two functions 
such that the degree $d_G(u)$ of every vertex $u$ of $G$
equals $\tau(u)+\kappa(u)$,
then 
a set $I$ of vertices of $G$ is $\kappa$-degenerate in $G$
if and only if 
$V(G)\setminus I$ is a dynamic monopoly of $(G,\tau)$.
This duality generalizes the well-known duality
between independent sets and vertex covers.

A function $\sigma\in\mathbb{N}_0^{V(G)}$ 
is a {\it partial incentive} of $(G,\tau)$~\cite{cogareva} if 
$H_{(G,\tau-\sigma)}(\emptyset)=V(G)$,
that is, reducing the initial thresholds $\tau$ 
as specified by $\sigma$,
the empty set becomes a dynamic monopoly.
Throughout this paper,
we define the {\it weight} of a function $f$ with domain $D$
as $f(D)=\sum\limits_{d\in D}f(d)$.
Let ${\rm pi}(G,\tau)$ denote the minimum weight
of a partial incentive of $(G,\tau)$,
and a partial incentive of $(G,\tau)$ of weight 
${\rm pi}(G,\tau)$ is called {\it optimal}.

\medskip
Our first contributions are complexity results.

We show NP-completeness for dynamic monopolies 
in planar graphs, which appears to be unknown. 

\begin{theorem} \label{thm:npdynplanar}
For a given triple $(G,\tau,k)$, 
where $G$ is a planar graph, 
$\tau$ is a threshold function for $G$, 
and $k$ is a positive integer, 
it is NP-complete to decide whether ${\rm dyn}(G,\tau)\leq k$.
\end{theorem}

Exploiting results from \cite{behelone,beehpera},
we establish the following hardness results for partial incentives.

\begin{theorem}\label{thm:hardpitw}
For a given pair $(G,\tau)$, 
where $G$ is a graph of order $n$ and treewidth $w$, and 
$\tau$ is a threshold function for $G$, 
an optimal partial incentive of $(G,\tau)$ 
cannot be computed in time $n^{o(\sqrt{w})}$ 
unless all problems in SNP can be solved in sub-exponential time.
\end{theorem}

\begin{theorem} \label{thm:nppichordal}
For a given triple $(G,\tau,k)$, 
where $G$ is a chordal graph, 
$\tau$ is a threshold function for $G$, and 
$k$ is a positive integer, 
it is NP-complete to decide whether ${\rm pi}(G,\tau)\leq k$.
\end{theorem}

Finally, using Theorem \ref{thm:npdynplanar}, we obtain the following.

\begin{corollary}\label{corollary:planar}
For a given triple $(G,\tau,k)$, where $G$ is a planar graph, 
$\tau$ is a threshold function for $G$, 
and $k$ is a positive integer, 
it is NP-complete to decide whether ${\rm pi}(G,\tau)\leq k$.
\end{corollary}
Our remaining contributions are algorithmic. 
We first show two tractable cases
for partial incentives.

\begin{theorem} \label{thm:pitw}
Let $w$ be a non-negative integer. 

For a given pair $(G,\tau)$,
where $G$ is a graph of order $n$ and treewidth at most $w$, 
and $\tau$ is a threshold function for $G$, 
an optimal partial incentive of $(G,\tau)$ 
can be computed in time $n^{O(w)}$.
\end{theorem}

\begin{theorem}\label{thm:piinterval}
Let $t$ be a non-negative integer. 

For a given pair $(G,\tau)$,
where $G$ is an interval graph of order $n$, 
and $\tau$ is a threshold function for $G$
with $\tau(u)\leq t$ for every vertex $u$ of $G$, 
an optimal partial incentive of $(G,\tau)$ 
can be determined in time $n^{O(t^2)}$.
\end{theorem}

Complementing our hardness result for dynamic monopolies in planar graphs, we contribute approximation algorithms.
We give a PTAS for degenerate sets in planar graph using Baker's layering technique~\cite{ba}.

\begin{theorem}\label{thm:ptaskappa}
Let $\epsilon$ be a positive real number.

For a given pair $(G,\kappa)$, 
where $G$ is a planar graph and $\kappa\in\mathbb{Z}^{V(G)}$,
a $\kappa$-degenerate set $I$ in $G$
with $|I|\geq (1-\epsilon) \alpha(G,\kappa)$ 
can be determined in polynomial time.
\end{theorem}
The main techniques for constructing PTASs on planar graphs such as Baker's layering technique~\cite{ba} and the bidimensionality theory~\cite{deha} seem not to work for finding a minimum dynamic monopoly, the dual problem of a maximum degenerate set.
On the positive side, we show the following.

\begin{theorem}\label{thm:apptw}
For a given pair $(G,\tau)$, where $G$ is a graph of order $n$ and size $m$, and $\tau$ is a threshold function for $G$, and for a given tree-decomposition of $G$ of width $w$ and size $O(n)$, a dynamic monopoly $D$ of $(G,\tau)$ with $|D|\leq (w+1){\rm dyn}(G,\tau)$ can be determined in time $O(n^2m)$.
\end{theorem}

\begin{corollary}\label{cor:appplanar}
For a given pair $(G,\tau)$, where $G$ is a planar graph of order $n$, and $\tau$ is a threshold function for $G$, a dynamic monopoly $D$ of $(G,\tau)$ with $|D|\leq O(\sqrt{n}){\rm dyn}(G,\tau)$ can be determined in time $O(n^3)$.
\end{corollary}

All proofs are given in the next section.

\section{Proofs}

Our hardness results follow from known results using two simple constructions described in Lemma~\ref{lemma1} and Lemma \ref{lemma2} below.

\begin{lemma}\label{lemma1}
Let $G$ be a graph of order $n$. 
Let $G'$ arise from $G$ by adding, for every edge $uv$ of $G$,
an independent set $I(uv)$ of order $n$ as well as all $2n$ possible edges between $\{ u,v\}$ and $I(uv)$.
Let 
$$\tau':V(G')\to\mathbb{Z}:
u\mapsto
\begin{cases}
d_G(u)n &\mbox{, if $u\in V(G)$, and}\\
1 &\mbox{, otherwise.}
\end{cases}
$$
\begin{enumerate}[(i)]
\item The minimum order of a vertex cover of $G$ equals ${\rm dyn}(G',\tau')$.
\item If $G$ is planar, then $G'$ is planar.
\end{enumerate}
\end{lemma}
\begin{proof} (i)
First, let $X$ be a vertex cover of $G$.
Let $H$ be the hull of $X$ in $(G',\tau')$.
Since every vertex in $V(G')\setminus V(G)$ has a neighbor in $X$,
and threshold value $1$, the set $H$ contains $V(G')\setminus V(G)$.
Therefore, for every vertex $u$ of $G'$ in $V(G)\setminus X$,
the set $H$ contains all $d_G(u)n$ neighbors of $u$ in $V(G')\setminus V(G)$,
which implies that $X$ is a dynamic monopoly of $(G',\tau')$.

Next, let $D$ be a dynamic monopoly of $(G',\tau')$.
Since replacing a vertex in $D\setminus V(G)$ by some neighbor in $V(G)$ 
yields a dynamic monopoly, we may assume that $D\subseteq V(G)$.
Suppose, for a contradiction, that $u,v\not\in D$ for some edge $uv$ in $G$.
Since $|N_{G'}(u)\setminus (\{ v\}\cup I(uv))|\leq (d_G(u)-1)(n+1)<d_G(u)n$
and $|N_{G'}(v)\setminus (\{ u\}\cup I(uv))|\leq (d_G(v)-1)(n+1)<d_G(v)n$,
we obtain a contradiction to the choice of $D$.
Hence, $D$ is a vertex cover of $G$.

This completes the proof of (i).\\[3mm]
(ii) This is obvious.
\end{proof}
Our first hardness result states the NP-completeness of dynamic monopolies for planar graphs,
which quite surprisingly seems to be unknown.

\begin{proof}[Proof of Theorem \ref{thm:npdynplanar}] The stated problem is clearly in NP,
and NP-completeness follows from the NP-completeness of {\sc Vertex Cover} for planar graphs \cite{gakost}
using Lemma \ref{lemma1}.
\end{proof}
For our remaining hardness results, 
we use the following lemma
relating dynamic monopolies with partial incentives.

\begin{lemma}\label{lemma2}
Let $G$ be a graph of order $n$, and let $\tau$ be a threshold function for $G$. 
Let $G'$ arise from $G$ by adding, for every vertex $u$ of $G$ with $\tau(u)>0$,
a path $P_u$ of order $\tau(u)$ as well as all $\tau(u)$ possible edges between $u$ and the vertices of $P_u$.
Let 
$$\tau':V(G')\to\mathbb{Z}:
u\mapsto
\begin{cases}
\tau(u)&\mbox{, if $u\in V(G)$, and}\\
1 &\mbox{, otherwise.}
\end{cases}
$$
\begin{enumerate}[(i)]
\item ${\rm dyn}(G,\tau)={\rm pi}(G',\tau')$.
\item If $G$ has treewidth $w$, then $G'$ has treewidth at most $\max\{ w,2\}$.
\item If $G$ is planar, then $G'$ is planar.
\item If $G$ is chordal, then $G'$ is chordal.
\end{enumerate}
\end{lemma}
\begin{proof}
(i) First, let $D$ be a dynamic monopoly for $(G,\tau)$, 
where we may assume that $\tau(u)>0$ for every vertex $u$ in $D$.
For every vertex $u$ in $D$, let $p_u$ be an endvertex of $P_u$.
Since 
$$\sigma:V(G')\to\mathbb{N}_0:
v\mapsto
\begin{cases}
1&\mbox{, if $v=p_u$ for some $u\in D$, and}\\
0 &\mbox{, otherwise,}
\end{cases}
$$
is a partial incentive of $(G',\tau')$, and $\sigma(V(G'))=|D|$,
we obtain ${\rm dyn}(G,\tau)\geq {\rm pi}(G',\tau')$.

Next, let $\sigma$ be a partial incentive of $(G',\tau')$.
In view of the definition of $\tau'$, 
we may assume that $\sigma(V(G))=0$, and $\sigma\big(V(P_u)\big)\leq 1$ for every vertex $u$ in $V(G)$.
Since $D=\big\{u\in V(G): \sigma\big(V(P_u)\big)=1 \big\}$
is a dynamic monopoly for $(G,\tau)$, and $|D|=\sigma(V(G'))$, 
we obtain ${\rm dyn}(G,\tau)\leq {\rm pi}(G',\tau')$.\\[3mm]
(ii) Let $\left(T,(X_t)_{t\in V(T)}\right)$ be a tree-decomposition of $G$,
cf. the beginning of Subsection \ref{subsec2.1}
Let $u$ be a vertex of $G$ with $\tau(u)=k>0$.
Let $P_u=v_1\ldots v_k$.
Let $t$ be a vertex of $T$ with $u\in X_t$. 
Attaching at $t$ within $T$ a path $t_1\ldots t_k$,
and setting 
$X_{t_1}=\{ u,v_1\},
X_{t_2}=\{ u,v_1,v_2\},
\ldots,
X_{t_k}=\{ u,v_{k-1},v_k\}$
yields a tree-decomposition of $G\cup P_u$ of width at most $\max\{ w,2\}$.
Proceeding similarly for all other vertices of $G$ with positive threshold value 
yields the desired statement,because
$\max\{ \max\{ w,2\},2\}=\max\{ w,2\}$.\\[3mm]
(iii) and (iv) are obvious.
\end{proof}
It is now easy to show the hardness of partial incentives
exploiting results from \cite{behelone,beehpera}.

\begin{proof}[Proof of Theorem \ref{thm:hardpitw}]
This follows easily from Theorem 1.2 in \cite{behelone}.
In fact, in order to compute an optimal dynamic monopoly for a given pair $(G,\tau)$,
where $G$ is a graph of order $n$, and $\tau$ is a threshold function for $G$,
a simple algorithmic reduction allows to assume that $\tau(u)<n$ for every vertex $u$ of $G$.
This implies that the graph $G'$ constructed in Lemma \ref{lemma2}
has order $n'$ at most $n^2$,
and an algorithm computing an optimal partial incentive of $(G',\tau')$ as in Lemma \ref{lemma2} 
in time $(n')^{o(\sqrt{w})}$,
would allow to compute an optimal dynamic monopoly of $(G,\tau)$ in time $n^{2o(\sqrt{w})}=n^{o(\sqrt{w})}$.
\end{proof}

\begin{proof}[Proof of Theorem \ref{thm:nppichordal}]
The stated problem is clearly in NP,
and NP-completeness follows from Theorem 1.3 in \cite{beehpera}
using Lemma \ref{lemma2} as in the previous proof.
\end{proof}
Corollary \ref{corollary:planar} follows easily 
from Theorem \ref{thm:npdynplanar}
and Lemma \ref{lemma2}.

We proceed to the two tractable cases.

\subsection{Partial incentives of graphs of bounded treewidth}\label{subsec2.1}

We need the notion of a nice tree-decomposition introduced by Kloks~\cite{kl}.
For a graph $G$, a {\it tree-decomposition of $G$} is a pair $\left(T,(X_t)_{t\in V(T)}\right)$, 
where $T$ is a tree and $(X_t)_{t\in V(T)}$ is a collection of sets of vertices of $G$ 
satisfying the following properties:
\begin{itemize}
\item $\bigcup\limits_{t\in V(T)} X_t=V(G)$,
\item for every edge $uv$ of $G$, there is a set $X_t$ containing both $u$ and $v$, and
\item for every vertex $u$ of $G$, the set $\{t\in V(T):u\in X_t\}$ induces a subtree of $T$.
\end{itemize}
The {\it width} of the tree-decomposition is $\max\limits_{t\in V(T)}|X_t|-1$,
and the {\it treewidth} of $G$ is the minimum width of a tree-decomposition of $G$.
For a rooted tree-decomposition $\left(T,(X_t)_{t\in V(T)}\right)$ and for every node $t$ of $T$, 
let $V_t$ denote the set of nodes of $T$ that contains $t$ and all its descendants,
and, let $G_t$ be the subgraph of $G$ induced by $\bigcup\limits_{s\in V_t} X_s$.
A tree-decomposition $\left(T,(X_t)_{t\in V(T)}\right)$ of $G$ is {\it nice} if $T$ is a rooted binary tree, 
and every node $t$ of $T$ is of one of the following types:
\begin{itemize}
\item $t$ is a leaf of $T$, and $X_t=\emptyset$ ({\it leaf node}).
\item $t$ has two children $t'$ and $t''$, and $X_t=X_{t'}=X_{t''}$ ({\it join node}).
\item $t$ has a unique child $t'$, and 
\\ either $|X_t\setminus X_{t'}|=1$ and $|X_{t'}\setminus X_t|=0$ ({\it introduce node}),
\\ or $|X_{t'}\setminus X_t|=1$ and $|X_t\setminus X_{t'}|=0$ ({\it forget node}).
\end{itemize}
For a linear order $\prec$ on a set $X$, and two subsets $X_1$ and $X_2$ of $X$,
we write $X_1\prec X_2$ if $x_1\prec x_2$ for every $x_1\in X_1$ and $x_2\in X_2$.

\begin{proof}[Proof of Theorem \ref{thm:pitw}]
Let $(G,\tau)$ be as in the statement of the theorem.

In view of the desired statement, we may assume that $G$ has treewidth exactly $w$.

We may assume that $\tau(u)<n$ for every vertex $u$ of $G$; 
otherwise, we compute 
an optimal partial incentive $\sigma'$ of $(G,\tau')$, 
where 
$$\tau'(u)=\tau(u)-\max\{\tau(u)-(n-1),0\}<n\mbox{  for every vertex $u$ of $G$},$$
and return 
the partial incentive $\sigma$ of $(G,\tau)$
with $\sigma(u)=\sigma'(u)+\max\{\tau(u)-(n-1),0\}$ for every vertex $u$ of $G$. 
It is easy to see that $\sigma$ is optimal.

In time $n^{O(w)}$ \cite{arcopr,kl}
we can determine a nice tree-decomposition $\left(T,(X_t)_{t\in V(T)}\right)$
of $G$ of width at most $w$ such that $n(T)=O(wn)$.
Let $r$ be the root of $T$.

Our approach is dynamic programming on the nice tree-decomposition and to propagate information in a bottom-up fashion within $T$.
For every node $t$ in $T$, we consider 
\begin{enumerate}[(i)]
\item all possible candidates $\sigma_t$ 
for the restriction of an optimal partial incentive $\sigma$ of $(G,\tau)$ to $X_t$, 
\item all possible linear orders $\prec_t$ in which the elements of $X_t$ may enter the hull $H_{(G,\tau-\sigma)}(\emptyset)$, and 
\item all possible amounts $\rho_t$ of help that each vertex in $X_t$ receives from outside of $G_t$ when it enters the hull.
\end{enumerate}
Accordingly, we define the notion of a {\it local cascade} $(\sigma_t,\prec_t,\rho_t)$ for $G_t$, where
\begin{enumerate}[(i)]
\item $\sigma_t:X_t\to [n-1]_0$,
\item $\prec_t$ is a linear order on $X_t$ with $\{u\in X_t:\tau (u)-\sigma_t(u)\leq 0\}\prec_t\{v\in X_t:\tau (v)-\sigma_t(v)> 0\}$, and
\item $\rho_t:X_t\to [n-1]_0$. 
\end{enumerate}
Since $|X_t|\leq w+1\leq n$, 
there are $O\left((w+1)!n^{O(w)}\right)=n^{O(w)}$ 
local cascades for $G_t$.

For a local cascade $(\sigma_t,\prec_t,\rho_t)$ for $G_t$, 
let $\pi_t(\sigma_t,\prec_t,\rho_t)$ 
be a function $\sigma\in [n-1]_0^{V(G_t)}$ 
of minimum cost such that $\sigma|_{X_t}=\sigma_t$, and
\begin{enumerate}[(iv)]
\item there is a linear extension $\prec$ of $\prec_t$ to $V(G_t)$ such that, 
for every vertex $u$ of $G_t$, \label{cond4}
\begin{itemize}
\item either $u\in V(G_t)\setminus X_t$ and $|N_{G_t}^\prec(u)|\geq\tau(u)-\sigma(u)$,
\item or $u\in X_t$ and $|N_{G_t}^\prec(u)|\geq\tau(u)-\sigma(u)-\rho_t(u)$,
\end{itemize}
\end{enumerate}
where $N_{G_t}^\prec(u)$ denotes the set of neighbors of $u$ in $G_t$ 
that appear before $u$ in the linear order $\prec$.
If no such function $\sigma$ exists, we set all the values of $\pi_t(\sigma_t,\prec_t,\rho_t)$ to $\infty$.
Note that $\pi_t(\sigma_t,\prec_t,\rho_t)$ is not necessarily unique.

Since $G_r=G$, and no vertex in $X_r$ has a neighbor outside of $V(G_r)$,
the following claim is obvious from the definitions.

\begin{claim}\label{claim1}
If $(\sigma_r,\prec_r,0)$ is a local cascade for $G_r$ 
that minimizes the cost of $\pi_r(\sigma_r,\prec_r,0)$,
then $\pi_r(\sigma_r,\prec_r,0)$ is an optimal partial incentive 
of $(G,\tau)$.
\end{claim}
We now explain how to compute $\pi_t(\sigma_t,\prec_t,\rho_t)$ 
for each node $t$ of $T$ using dynamic programming.

Since $X_t=\emptyset$ for every leaf node $t$, 
$\pi_t(\sigma_t,\prec_t,\rho_t)$
is initialized as an empty function for such a $t$.

\begin{claim}\label{claim2}
For an introduce or a forget node $t$ with child node $t'$, 
and for every local cascade $(\sigma_t,\prec_t,\rho_t)$ for $G_t$, 
given $\pi_{t'}(\sigma_{t'},\prec_{t'},\rho_{t'})$ 
for all local cascades $(\sigma_{t'},\prec_{t'},\rho_{t'})$ for $G_{t'}$, 
$\pi_{t}(\sigma_{t},\prec_{t},\rho_{t})$ 
can be computed in time $n^{O(w)}$.
\end{claim}
\begin{proof}[Proof of Claim~\ref{claim2}]
First, we assume that $t$ is an introduce node,
that is, $X_t\setminus X_{t'}=\{u\}$ for some vertex $u$ of $G$ that does not belong to $G_{t'}$.
Clearly, $N_{G_t}(u)\subseteq X_{t'}$.
If $|N_{G[X_t]}^{\prec_t}(u)|<\tau(u)-\sigma_t(u)-\rho_t(u)$, 
then no function satisfies condition (iv),
and we set all the values of $\pi_t(\sigma_t,\prec_t,\rho_t)$ to $\infty$.
Now, let $|N_{G[X_t]}^{\prec_t}(u)|\geq \tau(u)-\sigma_t(u)-\rho_t(u)$.
Every neighbor of $u$ in $X_{t'}$ that appears after $u$ in the linear order $\prec_t$
receives one additional unit of help from $u$ when it enters the hull.
Therefore, for every vertex $v\in X_{t'}$, let
$$\rho_{t'}(v)=
\begin{cases}
\rho_t(v)+1&\mbox{, if $uv\in E(G)$ and $u\prec_t v$, and}\\
\rho_t(v) &\mbox{, otherwise.}
\end{cases}
$$
Now, the function $\pi_t(\sigma_t,\prec_t,\rho_t)$ 
on $V(G_t)$ can be chosen as 
$$\pi_t(\sigma_t,\prec_t,\rho_t)(v)=
\begin{cases}
\sigma_t(u)&\mbox{, if $v=u$, and}\\
\pi_{t'}(\sigma_{t'},\prec_{t'},\rho_{t'})(v) &\mbox{, otherwise,}
\end{cases}
$$
where $\sigma_{t'}$ and $\prec_{t'}$ 
are the restrictions of $\sigma_t$ and $\prec_t$  to $X_{t'}$,
respectively.
Clearly, the computation of $\pi_t(\sigma_t,\prec_t,\rho_t)$ can be done in time $n^{O(w)}$.

Next, we assume that $t$ is a forget node, 
that is, $X_{t'}\setminus X_t=\{ u\}$ for some vertex $u$ of $G_t$,
and $G_t=G_{t'}$.
For every $i\in[n-1]_0$, let $\sigma_i$ be the function on $X_{t'}$ defined as
$$\sigma_i(x)=
\begin{cases}
i&\mbox{, if $x=u$, and}\\
\sigma_t(x) &\mbox{, otherwise.}
\end{cases}
$$
For every $j\in|X_{t'}|$, 
let $\prec_j$ be a linear extension of $\prec_t$ to $X_{t'}$ 
such that the vertex $u$ is the $j$th vertex in the order $\prec_j$.
Note that $v$ does not have any neighbors outside of $V(G_t)$.
Therefore, for every vertex $v\in X_{t'}$, let  
$$\rho_{t'}(v)=
\begin{cases}
0&\mbox{, if $v=u$, and}\\
\rho_t(v) &\mbox{, otherwise.}
\end{cases}
$$
Now, the function $\pi_t(\sigma_t,\prec_t,\rho_t)$ 
can be chosen as the function $\pi_{t'}(\sigma_i,\prec_j,\rho_{t'})$ 
that minimizes the cost 
among all choices of $i\in[n-1]_0$ and $j\in |X_{t'}|$.
Clearly, this can also be done in time $n^{O(w)}$.
\end{proof}

\begin{claim}\label{claim3}
For every join node $t$ with children nodes $t'$ and $t''$, 
and for every local cascade $(\sigma_t,\prec_t,\rho_t)$ for $G_t$, 
given $\pi_s(\sigma_s,\prec_s,\rho_s)$ 
for all local cascades $(\sigma_s,\prec_s,\rho_s)$ for $G_s$ and both $s\in\{t',t''\}$, 
$\pi_t(\sigma_t,\prec_t,\rho_t)$ can be computed in time $n^{O(w)}$.
\end{claim}
\begin{proof}[Proof of Claim~\ref{claim3}]
By the definition of a nice tree decomposition, 
$G_t=G_{t'}\cup G_{t''}$ 
and
$X_t=X_{t'}=X_{t''}=V(G_{t'})\cap V(G_{t''})$,
in particular, 
there are no edges between $V(G_{t'})\setminus X_t$ and $V(G_{t''})\setminus X_t$.

Let $\rho_{\rm int}\in [n-1]_0^{X_t}$ be such that 
$$\rho_{\rm int}(u)=
\max\left\{\tau(u)-\sigma_t(u)-\rho_t(u)-\left|N_{G[X_t]}^{\prec_t}(u)\right|,0\right\}$$
for every vertex $u$ in $X_t$, 
that is,  
$\rho_{\rm int}(u)$
is the (minimum) number of neighbors that $u$ needs in $V(G_t)\setminus X_t$ 
when entering the hull according to the local cascade 
$(\sigma_t,\prec_t,\rho_t)$.
Since these neighbors come from the two disjoint sets
$V(G_{t'})\setminus X_t$
and 
$V(G_{t''})\setminus X_t$,
say
$\rho'_{\rm int}(u)$ from the first set
and 
$\rho''_{\rm int}(u)$ from the second set,
it follows that 
$\pi_t(\sigma_t,\prec_t,\rho_t)$ can be chosen as
the common extension to $V(G_t)$ of the two functions
$\pi_{t'}(\sigma_t,\prec_t,\rho_t+\rho''_{\rm int})$
and
$\pi_{t''}(\sigma_t,\prec_t,\rho_t+\rho'_{\rm int})$,
where 
$\rho'_{\rm int},\rho''_{\rm int}\in [n-1]_0^{X_t}$
are chosen such that 
$\rho_{\rm int}=\rho'_{\rm int}+\rho''_{\rm int}$
and 
the cost of $\pi_t(\sigma_t,\prec_t,\rho_t)$
is minimized.
Clearly, this can be done in time $n^{O(w)}$.
\end{proof}
Since $T$ has order $O(wn)$,
the overall computation takes $n^{O(w)}$ time,
which completes the proof of Theorem~\ref{thm:pitw}.
\end{proof}

\subsection{Partial incentives of interval graphs with bounded thresholds}

Our approach is similar to the one in~\cite{beehpera} 
for computing a minimum dynamic monopoly 
of an interval graph with bounded thresholds.
We first adapt an auxiliary result from \cite{cedoperasz,chhuliwuye} 
to the setting of partial incentives.

\begin{lemma}\label{lem:tconnected}
Let $t$ be a non-negative integer.
Let $G$ be a $t$-connected chordal graph,
and let $\tau$ be a threshold function for $G$ 
with $\tau(u)\leq t$ for every vertex $u$ of $G$.
\begin{enumerate}[(i)]
\item For every clique $C$ in $G$ of order $t$ 
with vertices $v_0,v_1,\ldots,v_{t-1}$, the function 
$$\sigma(u)=
\begin{cases}
\tau(v_i)-i&\mbox{, if $u=v_i$ for some $i\in [t-1]_0$, and}\\
0 &\mbox{, otherwise.}
\end{cases}
$$
is a partial incentive of $(G,\tau)$.
\item ${\rm pi}(G,\tau)\leq {t+1\choose 2}$.
\end{enumerate}
\end{lemma}
\begin{proof}
(i) Clearly, $H=H_{(G,\tau-\sigma)}(\emptyset)$ contains $C$.
If $G$ is a clique, the bound on $\tau$ implies that $H$ contains all vertices of $G$.
If $G$ is not a clique, 
then $G$ has a simplicial vertex $u$ that does not belong to $C$.
Since $G-u$ is $t$-connected, it follows, by an inductive argument,
that $H$ contains $V(G)\setminus \{ u\}$,
and, since $d_G(u)\geq t$, $H$ contains also $u$,
that is, $\sigma$ is a partial incentive of $(G,\tau)$.\\[3mm]
(ii) This follows from the known fact that every $t$-connected chordal graph 
contains a clique of order $t$, and that 
$\sum\limits_{i=0}^{t-1}(t-i)={t+1\choose 2}$.
\end{proof}
Let $(G,\tau)$ be as in the statement of Theorem~\ref{thm:piinterval}.
We construct a sequence 
$G_1\subseteq G_2\subseteq \ldots \subseteq G_k$ 
of subgraphs of $G$ 
in such a way that $G_k=G$, and 
that Lemma~\ref{lem:tconnected} implies 
that the cost of every optimal partial incentive of $(G,\tau)$
within a suitable supergraph $\partial G_i$ 
of each graph $G_i-V(G_{i-1})$ is at most ${t+1\choose 2}$,
cf. condition (iv) below.
Clearly, we may assume that $G$ is connected.
Let $n$ be the order of $G$. 
In linear
time \cite{bolu}, we can determine an interval representation $(I(u))_{u\in V(G)}$ of $G$, that is, two distinct vertices $u$
and $v$ of $G$ are adjacent if and only if the intervals $I(u)$ and
$I(v)$ intersect. 
By applying simple manipulations, we may assume
that each interval $I(u)$ is closed, and that the $2n$ endpoints of
the $n$ intervals are all distinct.

Let $x_1<x_2<\ldots<x_{2n}$ be the endpoints of the intervals. For
each $i\in [2n-1]$, let $C_i$ be the set of vertices $u$ of $G$ with
$I_i:=[x_i,x_{i+1}]\subseteq I(u)$, and let $c_i=|C_i|$. Since each
$x_i$ is either the right endpoint of exactly one interval or the left
endpoint of exactly one interval, we have $|c_{i+1}-c_i|=1$ for every
$i\in [2n-1]$.

\begin{lemma}[Bessy et al.~\cite{beehpera}]\label{lem:oldclaim1}
If $C$ is a minimal vertex cut of $G$, then $C=C_i$ for some $i\in
[2n-2]\setminus \{ 1\}$ with $c_i<\min\{ c_{i-1},c_{i+1}\}$.
\end{lemma}
Let $j_1<j_2<\ldots<j_{k-1}$ be the indices $i$ in $[2n-1]\setminus \{
1\}$ with $c_i<\min\{ c_{i-1},c_{i+1},t\}$, and let $j_k=2n-1$. For
$i\in [k]$, let $G_i$ be the subgraph of $G$ induced by
$V_i:=C_1\cup\cdots\cup C_{j_i}$, and let $B_i=C_{j_i}$. Note that
$B_i$ contains all vertices in $V_i$ that have a neighbor in
$V(G)\setminus V_i$, and that $|B_i|<t$. Let $\partial V_1=V_1$, and,
for $i\in [k]\setminus \{ 1\}$, let $\partial V_i=(V_i\setminus
V_{i-1})\cup B_{i-1}$. For $i\in [k]$, let $\partial G_i$ be the
subgraph of $G$ induced by $\partial V_i$, cf Figure~\ref{fig:int-rep}.

\begin{figure}[htbp]
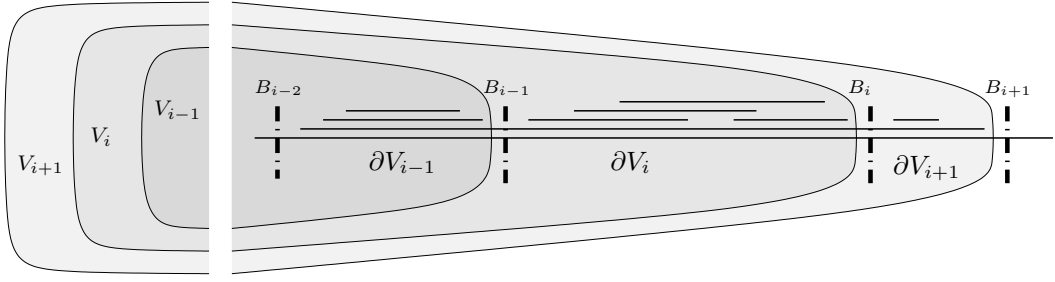

\begin{center}
\ifx\XFigwidth\undefined\dimen1=0pt\else\dimen1\XFigwidth\fi
\divide\dimen1 by 10374
\ifx\XFigheight\undefined\dimen3=0pt\else\dimen3\XFigheight\fi
\divide\dimen3 by 2724
\ifdim\dimen1=0pt\ifdim\dimen3=0pt\dimen1=2486sp\dimen3\dimen1
  \else\dimen1\dimen3\fi\else\ifdim\dimen3=0pt\dimen3\dimen1\fi\fi
\tikzpicture[x=+\dimen1, y=+\dimen3]
{\ifx\XFigu\undefined\catcode`\@11
\def\temp{\alloc@1\dimen\dimendef\insc@unt}\temp\XFigu\catcode`\@12\fi}
\XFigu2486sp
\ifdim\XFigu<0pt\XFigu-\XFigu\fi
\definecolor{xfigc32}{rgb}{0.612,0.000,0.000}
\definecolor{xfigc33}{rgb}{0.549,0.549,0.549}
\definecolor{xfigc34}{rgb}{0.549,0.549,0.549}
\definecolor{xfigc35}{rgb}{0.259,0.259,0.259}
\definecolor{xfigc36}{rgb}{0.549,0.549,0.549}
\definecolor{xfigc37}{rgb}{0.259,0.259,0.259}
\definecolor{xfigc38}{rgb}{0.549,0.549,0.549}
\definecolor{xfigc39}{rgb}{0.259,0.259,0.259}
\definecolor{xfigc40}{rgb}{0.549,0.549,0.549}
\definecolor{xfigc41}{rgb}{0.259,0.259,0.259}
\definecolor{xfigc42}{rgb}{0.549,0.549,0.549}
\definecolor{xfigc43}{rgb}{0.259,0.259,0.259}
\definecolor{xfigc44}{rgb}{0.557,0.557,0.557}
\definecolor{xfigc45}{rgb}{0.761,0.761,0.761}
\definecolor{xfigc46}{rgb}{0.431,0.431,0.431}
\definecolor{xfigc47}{rgb}{0.267,0.267,0.267}
\definecolor{xfigc48}{rgb}{0.557,0.561,0.557}
\definecolor{xfigc49}{rgb}{0.443,0.443,0.443}
\definecolor{xfigc50}{rgb}{0.682,0.682,0.682}
\definecolor{xfigc51}{rgb}{0.200,0.200,0.200}
\definecolor{xfigc52}{rgb}{0.580,0.576,0.584}
\definecolor{xfigc53}{rgb}{0.455,0.439,0.459}
\definecolor{xfigc54}{rgb}{0.333,0.333,0.333}
\definecolor{xfigc55}{rgb}{0.702,0.702,0.702}
\definecolor{xfigc56}{rgb}{0.765,0.765,0.765}
\definecolor{xfigc57}{rgb}{0.427,0.427,0.427}
\definecolor{xfigc58}{rgb}{0.271,0.271,0.271}
\definecolor{xfigc59}{rgb}{0.886,0.886,0.933}
\definecolor{xfigc60}{rgb}{0.580,0.580,0.604}
\definecolor{xfigc61}{rgb}{0.859,0.859,0.859}
\definecolor{xfigc62}{rgb}{0.631,0.631,0.718}
\definecolor{xfigc63}{rgb}{0.929,0.929,0.929}
\definecolor{xfigc64}{rgb}{0.878,0.878,0.878}
\definecolor{xfigc65}{rgb}{0.525,0.675,1.000}
\definecolor{xfigc66}{rgb}{0.439,0.439,1.000}
\definecolor{xfigc67}{rgb}{0.776,0.718,0.592}
\definecolor{xfigc68}{rgb}{0.937,0.973,1.000}
\definecolor{xfigc69}{rgb}{0.863,0.796,0.651}
\definecolor{xfigc70}{rgb}{0.251,0.251,0.251}
\definecolor{xfigc71}{rgb}{0.502,0.502,0.502}
\definecolor{xfigc72}{rgb}{0.753,0.753,0.753}
\definecolor{xfigc73}{rgb}{0.667,0.667,0.667}
\definecolor{xfigc74}{rgb}{0.780,0.765,0.780}
\definecolor{xfigc75}{rgb}{0.337,0.318,0.318}
\definecolor{xfigc76}{rgb}{0.843,0.843,0.843}
\definecolor{xfigc77}{rgb}{0.522,0.502,0.490}
\definecolor{xfigc78}{rgb}{0.824,0.824,0.824}
\definecolor{xfigc79}{rgb}{0.227,0.227,0.227}
\definecolor{xfigc80}{rgb}{0.271,0.451,0.667}
\definecolor{xfigc81}{rgb}{0.482,0.475,0.647}
\definecolor{xfigc82}{rgb}{0.451,0.459,0.549}
\definecolor{xfigc83}{rgb}{0.969,0.969,0.969}
\definecolor{xfigc84}{rgb}{0.255,0.271,0.255}
\definecolor{xfigc85}{rgb}{0.388,0.365,0.808}
\definecolor{xfigc86}{rgb}{0.745,0.745,0.745}
\definecolor{xfigc87}{rgb}{0.318,0.318,0.318}
\definecolor{xfigc88}{rgb}{0.906,0.890,0.906}
\definecolor{xfigc89}{rgb}{0.000,0.000,0.286}
\definecolor{xfigc90}{rgb}{0.475,0.475,0.475}
\definecolor{xfigc91}{rgb}{0.188,0.204,0.188}
\definecolor{xfigc92}{rgb}{0.255,0.255,0.255}
\definecolor{xfigc93}{rgb}{0.780,0.714,0.588}
\definecolor{xfigc94}{rgb}{0.867,0.616,0.576}
\definecolor{xfigc95}{rgb}{0.945,0.925,0.878}
\definecolor{xfigc96}{rgb}{0.886,0.784,0.659}
\definecolor{xfigc97}{rgb}{0.882,0.882,0.882}
\definecolor{xfigc98}{rgb}{0.855,0.478,0.102}
\definecolor{xfigc99}{rgb}{0.945,0.894,0.102}
\definecolor{xfigc100}{rgb}{0.533,0.490,0.761}
\definecolor{xfigc101}{rgb}{0.690,0.631,0.576}
\definecolor{xfigc102}{rgb}{0.514,0.486,0.867}
\definecolor{xfigc103}{rgb}{0.839,0.839,0.839}
\definecolor{xfigc104}{rgb}{0.549,0.549,0.647}
\definecolor{xfigc105}{rgb}{0.290,0.290,0.290}
\definecolor{xfigc106}{rgb}{0.549,0.420,0.420}
\definecolor{xfigc107}{rgb}{0.353,0.353,0.353}
\definecolor{xfigc108}{rgb}{0.388,0.388,0.388}
\definecolor{xfigc109}{rgb}{0.718,0.608,0.451}
\definecolor{xfigc110}{rgb}{0.255,0.576,1.000}
\definecolor{xfigc111}{rgb}{0.749,0.439,0.231}
\definecolor{xfigc112}{rgb}{0.859,0.467,0.000}
\definecolor{xfigc113}{rgb}{0.855,0.722,0.000}
\definecolor{xfigc114}{rgb}{0.000,0.392,0.000}
\definecolor{xfigc115}{rgb}{0.353,0.420,0.231}
\definecolor{xfigc116}{rgb}{0.827,0.827,0.827}
\definecolor{xfigc117}{rgb}{0.557,0.557,0.643}
\definecolor{xfigc118}{rgb}{0.953,0.725,0.365}
\definecolor{xfigc119}{rgb}{0.537,0.600,0.420}
\definecolor{xfigc120}{rgb}{0.392,0.392,0.392}
\definecolor{xfigc121}{rgb}{0.718,0.902,1.000}
\definecolor{xfigc122}{rgb}{0.525,0.753,0.925}
\definecolor{xfigc123}{rgb}{0.741,0.741,0.741}
\definecolor{xfigc124}{rgb}{0.827,0.584,0.322}
\definecolor{xfigc125}{rgb}{0.596,0.824,0.996}
\definecolor{xfigc126}{rgb}{0.380,0.380,0.380}
\definecolor{xfigc127}{rgb}{0.682,0.698,0.682}
\definecolor{xfigc128}{rgb}{1.000,0.604,0.000}
\definecolor{xfigc129}{rgb}{0.549,0.612,0.420}
\definecolor{xfigc130}{rgb}{0.969,0.420,0.000}
\definecolor{xfigc131}{rgb}{0.353,0.420,0.224}
\definecolor{xfigc132}{rgb}{0.549,0.612,0.420}
\definecolor{xfigc133}{rgb}{0.549,0.612,0.482}
\definecolor{xfigc134}{rgb}{0.094,0.290,0.094}
\definecolor{xfigc135}{rgb}{0.678,0.678,0.678}
\definecolor{xfigc136}{rgb}{0.969,0.741,0.353}
\definecolor{xfigc137}{rgb}{0.388,0.420,0.612}
\definecolor{xfigc138}{rgb}{0.871,0.000,0.000}
\definecolor{xfigc139}{rgb}{0.678,0.678,0.678}
\definecolor{xfigc140}{rgb}{0.969,0.741,0.353}
\definecolor{xfigc141}{rgb}{0.678,0.678,0.678}
\definecolor{xfigc142}{rgb}{0.969,0.741,0.353}
\definecolor{xfigc143}{rgb}{0.388,0.420,0.612}
\definecolor{xfigc144}{rgb}{0.322,0.420,0.161}
\definecolor{xfigc145}{rgb}{0.580,0.580,0.580}
\definecolor{xfigc146}{rgb}{0.000,0.388,0.000}
\definecolor{xfigc147}{rgb}{0.000,0.388,0.290}
\definecolor{xfigc148}{rgb}{0.482,0.518,0.290}
\definecolor{xfigc149}{rgb}{0.906,0.741,0.482}
\definecolor{xfigc150}{rgb}{0.647,0.710,0.776}
\definecolor{xfigc151}{rgb}{0.420,0.420,0.580}
\definecolor{xfigc152}{rgb}{0.518,0.420,0.420}
\definecolor{xfigc153}{rgb}{0.322,0.612,0.290}
\definecolor{xfigc154}{rgb}{0.839,0.906,0.906}
\definecolor{xfigc155}{rgb}{0.322,0.388,0.388}
\definecolor{xfigc156}{rgb}{0.094,0.420,0.290}
\definecolor{xfigc157}{rgb}{0.612,0.647,0.710}
\definecolor{xfigc158}{rgb}{1.000,0.580,0.000}
\definecolor{xfigc159}{rgb}{1.000,0.580,0.000}
\definecolor{xfigc160}{rgb}{0.000,0.388,0.290}
\definecolor{xfigc161}{rgb}{0.482,0.518,0.290}
\definecolor{xfigc162}{rgb}{0.388,0.451,0.482}
\definecolor{xfigc163}{rgb}{0.906,0.741,0.482}
\definecolor{xfigc164}{rgb}{0.094,0.290,0.094}
\definecolor{xfigc165}{rgb}{0.969,0.741,0.353}
\definecolor{xfigc166}{rgb}{0.000,0.000,0.000}
\definecolor{xfigc167}{rgb}{0.969,0.220,0.161}
\definecolor{xfigc168}{rgb}{0.000,0.000,0.000}
\definecolor{xfigc169}{rgb}{1.000,1.000,0.322}
\definecolor{xfigc170}{rgb}{0.322,0.475,0.290}
\definecolor{xfigc171}{rgb}{0.388,0.604,0.353}
\definecolor{xfigc172}{rgb}{0.776,0.380,0.259}
\definecolor{xfigc173}{rgb}{0.906,0.412,0.259}
\definecolor{xfigc174}{rgb}{1.000,0.475,0.322}
\definecolor{xfigc175}{rgb}{0.871,0.871,0.871}
\definecolor{xfigc176}{rgb}{0.953,0.933,0.827}
\definecolor{xfigc177}{rgb}{0.961,0.682,0.365}
\definecolor{xfigc178}{rgb}{0.584,0.808,0.600}
\definecolor{xfigc179}{rgb}{0.710,0.082,0.490}
\definecolor{xfigc180}{rgb}{0.933,0.933,0.933}
\definecolor{xfigc181}{rgb}{0.518,0.518,0.518}
\definecolor{xfigc182}{rgb}{0.482,0.482,0.482}
\definecolor{xfigc183}{rgb}{0.000,0.353,0.000}
\definecolor{xfigc184}{rgb}{0.906,0.451,0.451}
\definecolor{xfigc185}{rgb}{1.000,0.796,0.192}
\definecolor{xfigc186}{rgb}{0.161,0.475,0.290}
\definecolor{xfigc187}{rgb}{0.871,0.157,0.129}
\definecolor{xfigc188}{rgb}{0.129,0.349,0.776}
\definecolor{xfigc189}{rgb}{0.973,0.973,0.973}
\definecolor{xfigc190}{rgb}{0.902,0.902,0.902}
\definecolor{xfigc191}{rgb}{0.129,0.518,0.353}
\definecolor{xfigc192}{rgb}{0.906,0.906,0.906}
\definecolor{xfigc193}{rgb}{0.443,0.459,0.443}
\definecolor{xfigc194}{rgb}{0.851,0.851,0.851}
\definecolor{xfigc195}{rgb}{0.337,0.620,0.690}
\definecolor{xfigc196}{rgb}{0.788,0.788,0.788}
\definecolor{xfigc197}{rgb}{0.875,0.847,0.875}
\definecolor{xfigc198}{rgb}{0.969,0.953,0.969}
\definecolor{xfigc199}{rgb}{0.800,0.800,0.800}
\clip(-677,-4287) rectangle (9697,-1563);
\tikzset{inner sep=+0pt, outer sep=+0pt}
\pgfsetlinewidth{+7.5\XFigu}
\pgfsetfillcolor{.!5}
\filldraw (1575,-1575)--(1576,-1575)--(1578,-1575)--(1583,-1576)--(1590,-1576)--(1600,-1577)
  --(1615,-1579)--(1633,-1580)--(1656,-1583)--(1685,-1585)--(1719,-1588)--(1758,-1592)
  --(1804,-1596)--(1855,-1601)--(1913,-1607)--(1977,-1613)--(2048,-1619)--(2124,-1626)
  --(2207,-1634)--(2295,-1642)--(2389,-1651)--(2489,-1661)--(2593,-1670)--(2702,-1681)
  --(2815,-1691)--(2932,-1703)--(3052,-1714)--(3175,-1726)--(3301,-1738)--(3428,-1750)
  --(3558,-1762)--(3688,-1775)--(3819,-1787)--(3951,-1800)--(4082,-1812)--(4213,-1825)
  --(4343,-1838)--(4472,-1850)--(4600,-1863)--(4727,-1875)--(4851,-1887)--(4974,-1899)
  --(5094,-1911)--(5212,-1923)--(5328,-1934)--(5441,-1946)--(5551,-1957)--(5659,-1968)
  --(5764,-1978)--(5866,-1989)--(5966,-1999)--(6063,-2009)--(6157,-2019)--(6248,-2028)
  --(6337,-2038)--(6423,-2047)--(6507,-2056)--(6588,-2065)--(6666,-2073)--(6742,-2082)
  --(6816,-2090)--(6887,-2098)--(6956,-2106)--(7023,-2114)--(7088,-2121)--(7151,-2129)
  --(7212,-2136)--(7271,-2144)--(7329,-2151)--(7384,-2158)--(7438,-2165)--(7491,-2172)
  --(7542,-2179)--(7592,-2186)--(7640,-2192)--(7687,-2199)--(7733,-2206)--(7778,-2213)
  --(7866,-2226)--(7950,-2239)--(8031,-2253)--(8107,-2266)--(8180,-2279)--(8250,-2293)
  --(8316,-2306)--(8379,-2320)--(8439,-2333)--(8496,-2347)--(8550,-2360)--(8601,-2374)
  --(8649,-2388)--(8693,-2401)--(8735,-2415)--(8774,-2429)--(8810,-2443)--(8843,-2456)
  --(8873,-2470)--(8901,-2483)--(8926,-2497)--(8949,-2510)--(8969,-2523)--(8986,-2536)
  --(9002,-2549)--(9015,-2562)--(9027,-2574)--(9037,-2586)--(9045,-2599)--(9052,-2610)
  --(9058,-2622)--(9062,-2634)--(9066,-2645)--(9069,-2656)--(9071,-2667)--(9073,-2678)
  --(9074,-2689)--(9075,-2700)--(9077,-2721)--(9079,-2742)--(9080,-2763)--(9082,-2785)
  --(9083,-2807)--(9084,-2830)--(9085,-2853)--(9085,-2877)--(9086,-2901)--(9086,-2925)
  --(9086,-2949)--(9085,-2973)--(9085,-2997)--(9084,-3020)--(9083,-3043)--(9082,-3065)
  --(9080,-3087)--(9079,-3108)--(9077,-3129)--(9075,-3150)--(9074,-3161)--(9073,-3172)
  --(9071,-3183)--(9069,-3194)--(9066,-3205)--(9062,-3216)--(9058,-3228)--(9052,-3240)
  --(9045,-3251)--(9037,-3264)--(9027,-3276)--(9015,-3288)--(9002,-3301)--(8986,-3314)
  --(8969,-3327)--(8949,-3340)--(8926,-3353)--(8901,-3367)--(8873,-3380)--(8843,-3394)
  --(8810,-3407)--(8774,-3421)--(8735,-3435)--(8693,-3449)--(8649,-3462)--(8601,-3476)
  --(8550,-3490)--(8496,-3503)--(8439,-3517)--(8379,-3530)--(8316,-3544)--(8250,-3557)
  --(8180,-3571)--(8107,-3584)--(8031,-3597)--(7950,-3611)--(7866,-3624)--(7778,-3638)
  --(7733,-3644)--(7687,-3651)--(7640,-3658)--(7592,-3664)--(7542,-3671)--(7491,-3678)
  --(7438,-3685)--(7384,-3692)--(7329,-3699)--(7271,-3706)--(7212,-3714)--(7151,-3721)
  --(7088,-3729)--(7023,-3736)--(6956,-3744)--(6887,-3752)--(6816,-3760)--(6742,-3768)
  --(6666,-3777)--(6588,-3785)--(6507,-3794)--(6423,-3803)--(6337,-3812)--(6248,-3822)
  --(6157,-3831)--(6063,-3841)--(5966,-3851)--(5866,-3861)--(5764,-3872)--(5659,-3882)
  --(5551,-3893)--(5441,-3904)--(5328,-3916)--(5212,-3927)--(5094,-3939)--(4974,-3951)
  --(4851,-3963)--(4727,-3975)--(4600,-3987)--(4472,-4000)--(4343,-4012)--(4213,-4025)
  --(4082,-4038)--(3951,-4050)--(3819,-4063)--(3688,-4075)--(3558,-4088)--(3428,-4100)
  --(3301,-4112)--(3175,-4124)--(3052,-4136)--(2932,-4147)--(2815,-4159)--(2702,-4169)
  --(2593,-4180)--(2489,-4189)--(2389,-4199)--(2295,-4208)--(2207,-4216)--(2124,-4224)
  --(2048,-4231)--(1977,-4237)--(1913,-4243)--(1855,-4249)--(1804,-4254)--(1758,-4258)
  --(1719,-4262)--(1685,-4265)--(1656,-4267)--(1633,-4270)--(1615,-4271)--(1600,-4273)
  --(1590,-4274)--(1583,-4274)--(1578,-4275)--(1576,-4275)--(1575,-4275);
\filldraw (1350,-1575)--(1348,-1575)--(1343,-1575)--(1335,-1575)--(1321,-1575)--(1302,-1575)
  --(1277,-1575)--(1245,-1576)--(1207,-1576)--(1163,-1576)--(1113,-1576)--(1058,-1577)
  --(999,-1577)--(937,-1578)--(873,-1578)--(807,-1579)--(741,-1580)--(676,-1581)
  --(611,-1581)--(548,-1582)--(486,-1583)--(427,-1584)--(371,-1585)--(317,-1587)
  --(266,-1588)--(218,-1589)--(173,-1590)--(130,-1592)--(89,-1593)--(51,-1595)
  --(15,-1597)--(-18,-1599)--(-50,-1601)--(-81,-1603)--(-109,-1605)--(-137,-1607)
  --(-163,-1610)--(-188,-1613)--(-216,-1616)--(-243,-1619)--(-269,-1623)--(-294,-1628)
  --(-318,-1633)--(-341,-1638)--(-363,-1644)--(-384,-1651)--(-404,-1659)--(-424,-1667)
  --(-442,-1677)--(-460,-1687)--(-477,-1699)--(-492,-1711)--(-507,-1725)--(-521,-1739)
  --(-534,-1755)--(-546,-1773)--(-557,-1791)--(-567,-1810)--(-577,-1831)--(-585,-1853)
  --(-593,-1876)--(-600,-1900)--(-607,-1925)--(-612,-1952)--(-618,-1979)--(-622,-2008)
  --(-627,-2038)--(-631,-2070)--(-634,-2103)--(-638,-2138)--(-640,-2168)--(-643,-2200)
  --(-645,-2233)--(-647,-2267)--(-649,-2303)--(-651,-2340)--(-653,-2379)--(-655,-2419)
  --(-657,-2460)--(-658,-2503)--(-659,-2546)--(-661,-2591)--(-662,-2637)--(-663,-2684)
  --(-663,-2731)--(-664,-2779)--(-664,-2827)--(-665,-2876)--(-665,-2925)--(-665,-2974)
  --(-664,-3023)--(-664,-3071)--(-663,-3119)--(-663,-3166)--(-662,-3213)--(-661,-3259)
  --(-659,-3304)--(-658,-3347)--(-657,-3390)--(-655,-3431)--(-653,-3471)--(-651,-3510)
  --(-649,-3547)--(-647,-3583)--(-645,-3617)--(-643,-3650)--(-640,-3682)--(-638,-3713)
  --(-634,-3747)--(-631,-3780)--(-627,-3812)--(-622,-3842)--(-618,-3871)--(-612,-3898)
  --(-607,-3925)--(-600,-3950)--(-593,-3974)--(-585,-3997)--(-577,-4019)--(-567,-4040)
  --(-557,-4059)--(-546,-4077)--(-534,-4095)--(-521,-4111)--(-507,-4125)--(-492,-4139)
  --(-477,-4151)--(-460,-4163)--(-442,-4173)--(-424,-4183)--(-404,-4191)--(-384,-4199)
  --(-363,-4206)--(-341,-4212)--(-318,-4217)--(-294,-4222)--(-269,-4227)--(-243,-4231)
  --(-216,-4234)--(-188,-4238)--(-163,-4240)--(-137,-4243)--(-109,-4245)--(-81,-4247)
  --(-50,-4249)--(-18,-4251)--(15,-4253)--(51,-4255)--(89,-4257)--(130,-4258)
  --(173,-4260)--(218,-4261)--(266,-4262)--(317,-4263)--(371,-4265)--(427,-4266)
  --(486,-4267)--(548,-4268)--(611,-4269)--(676,-4269)--(741,-4270)--(807,-4271)
  --(873,-4272)--(937,-4272)--(999,-4273)--(1058,-4273)--(1113,-4274)--(1163,-4274)
  --(1207,-4274)--(1245,-4274)--(1277,-4275)--(1302,-4275)--(1321,-4275)--(1335,-4275)
  --(1343,-4275)--(1348,-4275)--(1350,-4275);
\pgfsetfillcolor{.!10}
\filldraw (1575,-1800)--(1576,-1800)--(1579,-1800)--(1584,-1801)--(1592,-1801)--(1603,-1802)
  --(1619,-1803)--(1640,-1805)--(1666,-1807)--(1697,-1809)--(1734,-1812)--(1777,-1816)
  --(1827,-1819)--(1883,-1824)--(1945,-1829)--(2013,-1834)--(2088,-1840)--(2169,-1846)
  --(2255,-1853)--(2347,-1860)--(2443,-1867)--(2544,-1875)--(2649,-1884)--(2758,-1892)
  --(2870,-1901)--(2984,-1910)--(3101,-1919)--(3219,-1929)--(3338,-1938)--(3457,-1948)
  --(3577,-1958)--(3697,-1967)--(3816,-1977)--(3934,-1987)--(4051,-1996)--(4166,-2006)
  --(4279,-2015)--(4390,-2024)--(4499,-2033)--(4606,-2042)--(4710,-2051)--(4811,-2060)
  --(4910,-2068)--(5006,-2077)--(5099,-2085)--(5189,-2093)--(5277,-2101)--(5362,-2108)
  --(5444,-2116)--(5523,-2123)--(5600,-2131)--(5674,-2138)--(5745,-2145)--(5814,-2151)
  --(5881,-2158)--(5946,-2165)--(6008,-2171)--(6068,-2178)--(6126,-2184)--(6182,-2190)
  --(6236,-2196)--(6288,-2202)--(6339,-2208)--(6388,-2214)--(6436,-2220)--(6482,-2226)
  --(6526,-2232)--(6570,-2238)--(6612,-2244)--(6653,-2250)--(6725,-2261)--(6794,-2272)
  --(6860,-2283)--(6923,-2294)--(6983,-2305)--(7040,-2316)--(7095,-2328)--(7147,-2339)
  --(7196,-2351)--(7243,-2363)--(7287,-2375)--(7329,-2387)--(7369,-2399)--(7406,-2412)
  --(7440,-2424)--(7472,-2437)--(7502,-2450)--(7530,-2462)--(7555,-2475)--(7578,-2488)
  --(7598,-2500)--(7617,-2513)--(7634,-2526)--(7649,-2538)--(7662,-2551)--(7673,-2563)
  --(7683,-2575)--(7691,-2587)--(7699,-2599)--(7704,-2611)--(7709,-2622)--(7713,-2634)
  --(7716,-2645)--(7719,-2656)--(7721,-2667)--(7723,-2678)--(7724,-2689)--(7725,-2700)
  --(7727,-2721)--(7729,-2742)--(7730,-2763)--(7732,-2785)--(7733,-2807)--(7734,-2830)
  --(7735,-2853)--(7735,-2877)--(7736,-2901)--(7736,-2925)--(7736,-2949)--(7735,-2973)
  --(7735,-2997)--(7734,-3020)--(7733,-3043)--(7732,-3065)--(7730,-3087)--(7729,-3108)
  --(7727,-3129)--(7725,-3150)--(7724,-3161)--(7723,-3172)--(7721,-3183)--(7719,-3194)
  --(7716,-3205)--(7713,-3216)--(7709,-3228)--(7704,-3239)--(7699,-3251)--(7691,-3263)
  --(7683,-3275)--(7673,-3287)--(7662,-3299)--(7649,-3312)--(7634,-3324)--(7617,-3337)
  --(7598,-3350)--(7578,-3362)--(7555,-3375)--(7530,-3388)--(7502,-3400)--(7472,-3413)
  --(7440,-3426)--(7406,-3438)--(7369,-3451)--(7329,-3463)--(7287,-3475)--(7243,-3487)
  --(7196,-3499)--(7147,-3511)--(7095,-3522)--(7040,-3534)--(6983,-3545)--(6923,-3556)
  --(6860,-3567)--(6794,-3578)--(6725,-3589)--(6653,-3600)--(6612,-3606)--(6570,-3612)
  --(6526,-3618)--(6482,-3624)--(6436,-3630)--(6388,-3636)--(6339,-3642)--(6288,-3648)
  --(6236,-3654)--(6182,-3660)--(6126,-3666)--(6068,-3672)--(6008,-3679)--(5946,-3685)
  --(5881,-3692)--(5814,-3699)--(5745,-3705)--(5674,-3712)--(5600,-3719)--(5523,-3727)
  --(5444,-3734)--(5362,-3742)--(5277,-3749)--(5189,-3757)--(5099,-3765)--(5006,-3773)
  --(4910,-3782)--(4811,-3790)--(4710,-3799)--(4606,-3808)--(4499,-3817)--(4390,-3826)
  --(4279,-3835)--(4166,-3844)--(4051,-3854)--(3934,-3863)--(3816,-3873)--(3697,-3883)
  --(3577,-3892)--(3457,-3902)--(3338,-3912)--(3219,-3921)--(3101,-3931)--(2984,-3940)
  --(2870,-3949)--(2758,-3958)--(2649,-3966)--(2544,-3975)--(2443,-3983)--(2347,-3990)
  --(2255,-3997)--(2169,-4004)--(2088,-4010)--(2013,-4016)--(1945,-4021)--(1883,-4026)
  --(1827,-4031)--(1777,-4034)--(1734,-4038)--(1697,-4041)--(1666,-4043)--(1640,-4045)
  --(1619,-4047)--(1603,-4048)--(1592,-4049)--(1584,-4049)--(1579,-4050)--(1576,-4050)
  --(1575,-4050);
\filldraw (1350,-1800)--(1347,-1800)--(1342,-1800)--(1331,-1800)--(1315,-1800)--(1293,-1800)
  --(1264,-1801)--(1228,-1801)--(1188,-1802)--(1142,-1802)--(1094,-1803)--(1042,-1804)
  --(990,-1804)--(937,-1805)--(886,-1806)--(835,-1808)--(787,-1809)--(741,-1810)
  --(698,-1812)--(658,-1813)--(620,-1815)--(585,-1817)--(552,-1819)--(521,-1821)
  --(493,-1823)--(467,-1826)--(442,-1828)--(418,-1831)--(396,-1834)--(375,-1838)
  --(353,-1841)--(333,-1845)--(313,-1850)--(294,-1855)--(275,-1861)--(257,-1868)
  --(240,-1875)--(223,-1883)--(208,-1892)--(192,-1902)--(178,-1913)--(164,-1926)
  --(151,-1939)--(138,-1954)--(127,-1970)--(116,-1987)--(106,-2006)--(97,-2025)
  --(88,-2046)--(80,-2068)--(73,-2091)--(66,-2115)--(60,-2141)--(55,-2167)--(50,-2195)
  --(45,-2225)--(41,-2255)--(38,-2288)--(35,-2315)--(32,-2345)--(29,-2375)--(27,-2407)
  --(24,-2440)--(22,-2475)--(20,-2511)--(18,-2548)--(17,-2586)--(15,-2626)--(14,-2666)
  --(13,-2708)--(12,-2750)--(11,-2793)--(11,-2837)--(10,-2881)--(10,-2925)--(10,-2969)
  --(11,-3013)--(11,-3057)--(12,-3100)--(13,-3142)--(14,-3184)--(15,-3224)--(17,-3264)
  --(18,-3302)--(20,-3339)--(22,-3375)--(24,-3410)--(27,-3443)--(29,-3475)--(32,-3505)
  --(35,-3535)--(38,-3563)--(41,-3595)--(45,-3625)--(50,-3655)--(55,-3683)--(60,-3709)
  --(66,-3735)--(73,-3759)--(80,-3782)--(88,-3804)--(97,-3825)--(106,-3844)--(116,-3863)
  --(127,-3880)--(138,-3896)--(151,-3911)--(164,-3924)--(178,-3937)--(192,-3948)
  --(208,-3958)--(223,-3967)--(240,-3975)--(257,-3982)--(275,-3989)--(294,-3995)
  --(313,-4000)--(333,-4005)--(353,-4009)--(375,-4013)--(396,-4016)--(418,-4019)
  --(442,-4022)--(467,-4024)--(493,-4027)--(521,-4029)--(552,-4031)--(585,-4033)
  --(620,-4035)--(658,-4037)--(698,-4038)--(741,-4040)--(787,-4041)--(835,-4042)
  --(886,-4044)--(937,-4045)--(990,-4046)--(1042,-4046)--(1094,-4047)--(1142,-4048)
  --(1188,-4048)--(1228,-4049)--(1264,-4049)--(1293,-4050)--(1315,-4050)--(1331,-4050)
  --(1342,-4050)--(1347,-4050)--(1350,-4050);
\pgfsetfillcolor{.!15}
\filldraw (1575,-2025)--(1577,-2025)--(1581,-2026)--(1589,-2026)--(1601,-2028)--(1619,-2029)
  --(1642,-2032)--(1672,-2035)--(1708,-2038)--(1751,-2042)--(1800,-2047)--(1856,-2053)
  --(1916,-2059)--(1982,-2066)--(2051,-2073)--(2124,-2080)--(2200,-2088)--(2276,-2096)
  --(2354,-2104)--(2432,-2113)--(2509,-2121)--(2585,-2130)--(2660,-2138)--(2732,-2146)
  --(2802,-2154)--(2870,-2162)--(2934,-2170)--(2996,-2177)--(3055,-2185)--(3112,-2192)
  --(3165,-2199)--(3216,-2206)--(3264,-2213)--(3309,-2220)--(3352,-2227)--(3393,-2233)
  --(3431,-2240)--(3468,-2247)--(3503,-2253)--(3536,-2260)--(3567,-2267)--(3597,-2274)
  --(3625,-2281)--(3653,-2288)--(3693,-2299)--(3731,-2310)--(3767,-2322)--(3801,-2334)
  --(3833,-2347)--(3863,-2360)--(3891,-2374)--(3918,-2388)--(3942,-2402)--(3965,-2417)
  --(3986,-2433)--(4005,-2448)--(4022,-2464)--(4038,-2480)--(4052,-2496)--(4064,-2513)
  --(4075,-2529)--(4085,-2545)--(4093,-2562)--(4099,-2578)--(4105,-2594)--(4110,-2609)
  --(4114,-2625)--(4117,-2640)--(4120,-2655)--(4122,-2670)--(4123,-2685)--(4125,-2700)
  --(4127,-2721)--(4129,-2742)--(4130,-2763)--(4132,-2785)--(4133,-2807)--(4134,-2830)
  --(4135,-2853)--(4135,-2877)--(4136,-2901)--(4136,-2925)--(4136,-2949)--(4135,-2973)
  --(4135,-2997)--(4134,-3020)--(4133,-3043)--(4132,-3065)--(4130,-3087)--(4129,-3108)
  --(4127,-3129)--(4125,-3150)--(4123,-3165)--(4122,-3180)--(4120,-3195)--(4117,-3210)
  --(4114,-3225)--(4110,-3241)--(4105,-3256)--(4099,-3272)--(4093,-3288)--(4085,-3305)
  --(4075,-3321)--(4064,-3337)--(4052,-3354)--(4038,-3370)--(4022,-3386)--(4005,-3402)
  --(3986,-3417)--(3965,-3433)--(3942,-3448)--(3918,-3462)--(3891,-3476)--(3863,-3490)
  --(3833,-3503)--(3801,-3516)--(3767,-3528)--(3731,-3540)--(3693,-3551)--(3653,-3563)
  --(3625,-3569)--(3597,-3576)--(3567,-3583)--(3536,-3590)--(3503,-3597)--(3468,-3603)
  --(3431,-3610)--(3393,-3617)--(3352,-3623)--(3309,-3630)--(3264,-3637)--(3216,-3644)
  --(3165,-3651)--(3112,-3658)--(3055,-3665)--(2996,-3673)--(2934,-3680)--(2870,-3688)
  --(2802,-3696)--(2732,-3704)--(2660,-3712)--(2585,-3720)--(2509,-3729)--(2432,-3737)
  --(2354,-3746)--(2276,-3754)--(2200,-3762)--(2124,-3770)--(2051,-3777)--(1982,-3784)
  --(1916,-3791)--(1856,-3797)--(1800,-3803)--(1751,-3808)--(1708,-3812)--(1672,-3815)
  --(1642,-3818)--(1619,-3821)--(1601,-3822)--(1589,-3824)--(1581,-3824)--(1577,-3825)
  --(1575,-3825);
\filldraw (1350,-2025)--(1347,-2025)--(1339,-2025)--(1327,-2025)--(1308,-2026)--(1284,-2027)
  --(1256,-2027)--(1225,-2029)--(1192,-2030)--(1160,-2031)--(1129,-2033)--(1100,-2035)
  --(1074,-2037)--(1049,-2040)--(1027,-2043)--(1006,-2046)--(987,-2049)--(970,-2053)
  --(953,-2058)--(938,-2063)--(924,-2067)--(910,-2073)--(897,-2079)--(884,-2086)
  --(871,-2094)--(858,-2103)--(845,-2114)--(833,-2125)--(821,-2138)--(809,-2153)
  --(798,-2169)--(787,-2186)--(777,-2205)--(767,-2225)--(758,-2247)--(750,-2270)
  --(742,-2294)--(735,-2320)--(729,-2347)--(723,-2376)--(717,-2406)--(713,-2438)
  --(709,-2462)--(706,-2488)--(703,-2515)--(700,-2543)--(698,-2573)--(696,-2604)
  --(694,-2636)--(692,-2669)--(690,-2703)--(689,-2739)--(687,-2775)--(687,-2812)
  --(686,-2849)--(686,-2887)--(685,-2925)--(686,-2963)--(686,-3001)--(687,-3038)
  --(687,-3075)--(689,-3111)--(690,-3147)--(692,-3181)--(694,-3214)--(696,-3246)
  --(698,-3277)--(700,-3307)--(703,-3335)--(706,-3362)--(709,-3388)--(713,-3413)
  --(717,-3444)--(723,-3474)--(729,-3503)--(735,-3530)--(742,-3556)--(750,-3580)
  --(758,-3603)--(767,-3625)--(777,-3645)--(787,-3664)--(798,-3681)--(809,-3697)
  --(821,-3712)--(833,-3725)--(845,-3736)--(858,-3747)--(871,-3756)--(884,-3764)
  --(897,-3771)--(910,-3777)--(924,-3783)--(938,-3788)--(953,-3792)--(970,-3797)
  --(987,-3801)--(1006,-3804)--(1027,-3807)--(1049,-3810)--(1074,-3813)--(1100,-3815)
  --(1129,-3817)--(1160,-3819)--(1192,-3820)--(1225,-3821)--(1256,-3823)--(1284,-3823)
  --(1308,-3824)--(1327,-3825)--(1339,-3825)--(1347,-3825)--(1350,-3825);
\pgfsetlinewidth{+15\XFigu}
\draw (1800,-2925)--(9675,-2925);
\draw (2250,-2835)--(9000,-2835);
\draw (2475,-2745)--(4050,-2745);
\draw (2700,-2655)--(3825,-2655);
\draw (4500,-2745)--(6075,-2745);
\draw (4950,-2655)--(6750,-2655);
\draw (6525,-2745)--(7650,-2745);
\draw (5400,-2565)--(7425,-2565);
\draw (8100,-2745)--(8550,-2745);
\pgfsetlinewidth{+45\XFigu}
\pgfsetdash{{+150\XFigu}{+75\XFigu}{+15\XFigu}{+75\XFigu}}{+0pt}
\draw (4275,-2610)--(4275,-3375);
\draw (2025,-2610)--(2025,-3330);
\draw (7875,-2610)--(7875,-3375);
\draw (9225,-2610)--(9225,-3375);
\pgfsetfillcolor{.}
\pgftext[base,left,at=\pgfqpointxy{1800}{-2475}] {\fontsize{7}{8.4}\usefont{T1}{ptm}{m}{n}$B_{i-2}$};
\pgftext[base,left,at=\pgfqpointxy{4050}{-2475}] {\fontsize{7}{8.4}\usefont{T1}{ptm}{m}{n}$B_{i-1}$};
\pgftext[base,left,at=\pgfqpointxy{7650}{-2475}] {\fontsize{7}{8.4}\usefont{T1}{ptm}{m}{n}$B_{i}$};
\pgftext[base,left,at=\pgfqpointxy{9000}{-2475}] {\fontsize{7}{8.4}\usefont{T1}{ptm}{m}{n}$B_{i+1}$};
\pgftext[base,left,at=\pgfqpointxy{2925}{-3240}] {\fontsize{10}{12}\usefont{T1}{ptm}{m}{n}$\partial V_{i-1}$};
\pgftext[base,left,at=\pgfqpointxy{5310}{-3240}] {\fontsize{10}{12}\usefont{T1}{ptm}{m}{n}$\partial V_{i}$};
\pgftext[base,left,at=\pgfqpointxy{8100}{-3285}] {\fontsize{10}{12}\usefont{T1}{ptm}{m}{n}$\partial V_{i+1}$};
\pgftext[base,left,at=\pgfqpointxy{810}{-2700}] {\fontsize{8}{9.6}\usefont{T1}{ptm}{m}{n}$V_{i-1}$};
\pgftext[base,left,at=\pgfqpointxy{180}{-2970}] {\fontsize{8}{9.6}\usefont{T1}{ptm}{m}{n}$V_{i}$};
\pgftext[base,left,at=\pgfqpointxy{-540}{-3240}] {\fontsize{8}{9.6}\usefont{T1}{ptm}{m}{n}$V_{i+1}$};
\endtikzpicture%
\caption{Sets $B_i$, $V_i$ and $\partial V_i$ on the interval
  representation of $G$ (for instance, $B_i$ contains all the
  intervals crossing the corresponding dotted line, $\partial V_i$
  contains all the intervals intersecting the zone between $B_{i-1}$
  and $B_i$, and $V_i$ contains all the intervals intersecting the
  corresponding zone).}\label{fig:int-rep}
\end{center}  
\end{figure}

\begin{lemma}[Bessy et al.~\cite{beehpera}]\label{lem:oldclaim2}
For every $i\in [k]$, the graph $\partial G_i$ is either a clique of
order at most $t$ or a $t$-connected graph.
\end{lemma}
\begin{proof}[Proof of Theorem~\ref{thm:piinterval}]
We use the same notation as above.
Our approach is dynamic programming on the 
sequence $G_1\subseteq G_2\subseteq \ldots \subseteq G_k$.
Similarly as for Theorem~\ref{thm:pitw}, 
we introduce the notion of a {\it local cascade}.
More precisely, for every subgraph $G_i$, 
we consider 
\begin{enumerate}[(i)]
\item all possible candidates for the restriction $\sigma_i$ 
of an optimal partial incentive to $B_i$, 
\item all possible linear orders, in which the elements of $B_i$ enter the hull, and 
\item all possible amounts of help that each vertex in $B_i$ 
receives from outside of $G_i$ when it enters the hull.
\end{enumerate}
Formally, 
a {\it local cascade} for $G_i$ is a triple $(\sigma_i,\prec_i,\rho_i)$, where
\begin{enumerate}[(i)]
\item $\sigma_i:B_i\to [t]_0$,
\item $\prec_i$ is a linear order on $B_i$ with $\{u\in B_i:\tau (u)-\sigma(u)\leq 0\}\prec_i\{v\in B_i:\tau (v)-\sigma(v)> 0\}$,
\item $\rho_i:B_i\to [n-1]_0$. 
\end{enumerate}
Since $|B_i|<t$, there are $O\left((t+1)^{t-1}(t-1)!n^{t-1}\right)=n^{O(t)}$ local cascades for $G_i$.

For a local cascade $(\sigma_i,\prec_i,\rho_i)$ for $G_i$, let $\pi_i(\sigma_i,\prec_i,\rho_i)$ be a function $\sigma\in [t]_0^{V(G_i)}$ with minimum cost such that $\sigma|_{B_i}=\sigma_i$,
\begin{enumerate}
\item[(iv)] $\sum\limits_{v\in \partial V_j} \sigma(v)\leq {t+1\choose 2}$ for every $j\in[i]$, and
\item[(v)] there is a linear extension $\prec$ of $\prec_i$ to $G_i$ such that for every $u\in V(G_i)$
\begin{itemize}
\item either $u\in V(G_i)\setminus B_i$ and $|N_{G_i}^\prec(u)|\geq\tau(u)-\sigma(u)$,
\item or $u\in B_i$ and $|N_{G_i}^\prec(u)|\geq\tau(u)-\sigma(u)-\rho_i(u)$,
\end{itemize}
\end{enumerate}
where $N_{G_i}^\prec(u)$ denotes the neighbors of $u$ in $G_i$ 
that appear before $u$ in the linear order $\prec$.
If no such function $\sigma$ exists, 
then we set all the values of $\pi_i(\sigma_i,\prec_i,\rho_i)$ to $\infty$.
Otherwise, it is clear from the above definitions 
that $\pi_i(\sigma_i,\prec_i,\rho_i)$ 
is a partial incentive of $(G_i,\tau)$ 
which is optimal with respect to the conditions 
imposed by $(\sigma_i,\prec_i,\rho_i)$.

Note that $B_k$ consists of a single vertex, 
which does not have any neighbors outside of $G_k=G$.
Therefore, using Lemma \ref{lem:tconnected}(ii) and Lemma \ref{lem:oldclaim2},
the following claim follow immediately from the definitions.

\begin{claim}\label{claim3.1}
The function $\pi_k(\sigma_k,\emptyset,0)$ 
that minimizes the cost among all local cascades $(\sigma_k,\emptyset,0)$ for $G_k$ 
is an optimal partial incentive of $(G,\tau)$. 
\end{claim}
The next two claims show how 
$\pi_i(\sigma_i,\prec_i,\rho_i)$ can be computed recursively.
\begin{claim}\label{claim3.2}
For every local cascade $(\sigma_1,\prec_1,\rho_1)$ for $G_1$, 
$\pi_1(\sigma_1,\prec_1,\rho_1)$ can be computed in time $n^{O(t^2)}$.
\end{claim}
\begin{proof}[Proof of Claim~\ref{claim3.2}]
Since $V_1=\partial V_1$, 
there are $O(n^{t(t+1)/2})$ 
candidates for a function $\sigma\in [t]_0^{V(G_1)}$
satisfying condition (iv).
Hence, we can compute a function 
$\pi_1(\sigma_1,\prec_1,\rho_1)$
by brute-force in time $n^{O(t^2)}$.
\end{proof}

\begin{claim}\label{claim3.3}
For every $i\in [k]\setminus \{ 1\}$, 
and every local cascade $(\sigma_i,\prec_i,\rho_i)$ for $G_i$,
given $\pi_{i-1}(\sigma_{i-1},\prec_{i-1}, \rho_{i-1})$
for all local cascades $(\sigma_{i-1},\prec_{i-1},\rho_{i-1})$ for $G_{i-1}$,
$\pi_i(\sigma_i,\prec_i,\rho_i)$
can be computed in time $n^{O(t^2)}$.
\end{claim}
\begin{proof}[Proof of Claim~\ref{claim3.3}]
By definition, we have $B_i\cap V_{i-1}\subseteq B_{i-1}$. Therefore,
the two sets $B_{i-1}'=B_i\cap V_{i-1}$ and
$B''_{i-1}=B_{i-1}\setminus B'_{i-1}$ partition the set $B_{i-1}$.
Note that $B_{i-1}'=B_i\cap B_{i-1}$,
and $B''_{i-1}=B_{i-1}\setminus B_i$,
cf. Figure~\ref{fig1}.

\begin{figure}[hbt]
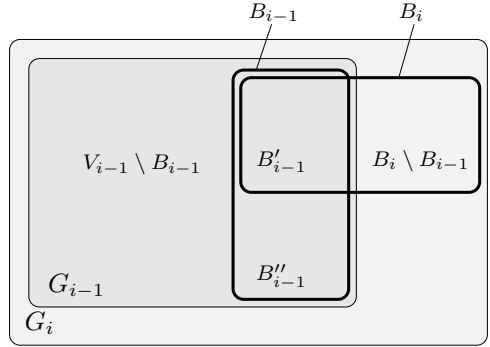

\begin{center}

\ifx\XFigwidth\undefined\dimen1=0pt\else\dimen1\XFigwidth\fi
\divide\dimen1 by 5604
\ifx\XFigheight\undefined\dimen3=0pt\else\dimen3\XFigheight\fi
\divide\dimen3 by 4041
\ifdim\dimen1=0pt\ifdim\dimen3=0pt\dimen1=2100sp\dimen3\dimen1
\else\dimen1\dimen3\fi\else\ifdim\dimen3=0pt\dimen3\dimen1\fi\fi
\tikzpicture[x=+\dimen1, y=+\dimen3]
{\ifx\XFigu\undefined\catcode`\@11
\def\temp{\alloc@1\dimen\dimendef\insc@unt}\temp\XFigu\catcode`\@12\fi}
\XFigu2486sp
\ifdim\XFigu<0pt\XFigu-\XFigu\fi
\definecolor{xfigc32}{rgb}{0.612,0.000,0.000}
\definecolor{xfigc33}{rgb}{0.549,0.549,0.549}
\definecolor{xfigc34}{rgb}{0.549,0.549,0.549}
\definecolor{xfigc35}{rgb}{0.259,0.259,0.259}
\definecolor{xfigc36}{rgb}{0.549,0.549,0.549}
\definecolor{xfigc37}{rgb}{0.259,0.259,0.259}
\definecolor{xfigc38}{rgb}{0.549,0.549,0.549}
\definecolor{xfigc39}{rgb}{0.259,0.259,0.259}
\definecolor{xfigc40}{rgb}{0.549,0.549,0.549}
\definecolor{xfigc41}{rgb}{0.259,0.259,0.259}
\definecolor{xfigc42}{rgb}{0.549,0.549,0.549}
\definecolor{xfigc43}{rgb}{0.259,0.259,0.259}
\definecolor{xfigc44}{rgb}{0.557,0.557,0.557}
\definecolor{xfigc45}{rgb}{0.761,0.761,0.761}
\definecolor{xfigc46}{rgb}{0.431,0.431,0.431}
\definecolor{xfigc47}{rgb}{0.267,0.267,0.267}
\definecolor{xfigc48}{rgb}{0.557,0.561,0.557}
\definecolor{xfigc49}{rgb}{0.443,0.443,0.443}
\definecolor{xfigc50}{rgb}{0.682,0.682,0.682}
\definecolor{xfigc51}{rgb}{0.200,0.200,0.200}
\definecolor{xfigc52}{rgb}{0.580,0.576,0.584}
\definecolor{xfigc53}{rgb}{0.455,0.439,0.459}
\definecolor{xfigc54}{rgb}{0.333,0.333,0.333}
\definecolor{xfigc55}{rgb}{0.702,0.702,0.702}
\definecolor{xfigc56}{rgb}{0.765,0.765,0.765}
\definecolor{xfigc57}{rgb}{0.427,0.427,0.427}
\definecolor{xfigc58}{rgb}{0.271,0.271,0.271}
\definecolor{xfigc59}{rgb}{0.886,0.886,0.933}
\definecolor{xfigc60}{rgb}{0.580,0.580,0.604}
\definecolor{xfigc61}{rgb}{0.859,0.859,0.859}
\definecolor{xfigc62}{rgb}{0.631,0.631,0.718}
\definecolor{xfigc63}{rgb}{0.929,0.929,0.929}
\definecolor{xfigc64}{rgb}{0.878,0.878,0.878}
\definecolor{xfigc65}{rgb}{0.525,0.675,1.000}
\definecolor{xfigc66}{rgb}{0.439,0.439,1.000}
\definecolor{xfigc67}{rgb}{0.776,0.718,0.592}
\definecolor{xfigc68}{rgb}{0.937,0.973,1.000}
\definecolor{xfigc69}{rgb}{0.863,0.796,0.651}
\definecolor{xfigc70}{rgb}{0.251,0.251,0.251}
\definecolor{xfigc71}{rgb}{0.502,0.502,0.502}
\definecolor{xfigc72}{rgb}{0.753,0.753,0.753}
\definecolor{xfigc73}{rgb}{0.667,0.667,0.667}
\definecolor{xfigc74}{rgb}{0.780,0.765,0.780}
\definecolor{xfigc75}{rgb}{0.337,0.318,0.318}
\definecolor{xfigc76}{rgb}{0.843,0.843,0.843}
\definecolor{xfigc77}{rgb}{0.522,0.502,0.490}
\definecolor{xfigc78}{rgb}{0.824,0.824,0.824}
\definecolor{xfigc79}{rgb}{0.227,0.227,0.227}
\definecolor{xfigc80}{rgb}{0.271,0.451,0.667}
\definecolor{xfigc81}{rgb}{0.482,0.475,0.647}
\definecolor{xfigc82}{rgb}{0.451,0.459,0.549}
\definecolor{xfigc83}{rgb}{0.969,0.969,0.969}
\definecolor{xfigc84}{rgb}{0.255,0.271,0.255}
\definecolor{xfigc85}{rgb}{0.388,0.365,0.808}
\definecolor{xfigc86}{rgb}{0.745,0.745,0.745}
\definecolor{xfigc87}{rgb}{0.318,0.318,0.318}
\definecolor{xfigc88}{rgb}{0.906,0.890,0.906}
\definecolor{xfigc89}{rgb}{0.000,0.000,0.286}
\definecolor{xfigc90}{rgb}{0.475,0.475,0.475}
\definecolor{xfigc91}{rgb}{0.188,0.204,0.188}
\definecolor{xfigc92}{rgb}{0.255,0.255,0.255}
\definecolor{xfigc93}{rgb}{0.780,0.714,0.588}
\definecolor{xfigc94}{rgb}{0.867,0.616,0.576}
\definecolor{xfigc95}{rgb}{0.945,0.925,0.878}
\definecolor{xfigc96}{rgb}{0.886,0.784,0.659}
\definecolor{xfigc97}{rgb}{0.882,0.882,0.882}
\definecolor{xfigc98}{rgb}{0.855,0.478,0.102}
\definecolor{xfigc99}{rgb}{0.945,0.894,0.102}
\definecolor{xfigc100}{rgb}{0.533,0.490,0.761}
\definecolor{xfigc101}{rgb}{0.690,0.631,0.576}
\definecolor{xfigc102}{rgb}{0.514,0.486,0.867}
\definecolor{xfigc103}{rgb}{0.839,0.839,0.839}
\definecolor{xfigc104}{rgb}{0.549,0.549,0.647}
\definecolor{xfigc105}{rgb}{0.290,0.290,0.290}
\definecolor{xfigc106}{rgb}{0.549,0.420,0.420}
\definecolor{xfigc107}{rgb}{0.353,0.353,0.353}
\definecolor{xfigc108}{rgb}{0.388,0.388,0.388}
\definecolor{xfigc109}{rgb}{0.718,0.608,0.451}
\definecolor{xfigc110}{rgb}{0.255,0.576,1.000}
\definecolor{xfigc111}{rgb}{0.749,0.439,0.231}
\definecolor{xfigc112}{rgb}{0.859,0.467,0.000}
\definecolor{xfigc113}{rgb}{0.855,0.722,0.000}
\definecolor{xfigc114}{rgb}{0.000,0.392,0.000}
\definecolor{xfigc115}{rgb}{0.353,0.420,0.231}
\definecolor{xfigc116}{rgb}{0.827,0.827,0.827}
\definecolor{xfigc117}{rgb}{0.557,0.557,0.643}
\definecolor{xfigc118}{rgb}{0.953,0.725,0.365}
\definecolor{xfigc119}{rgb}{0.537,0.600,0.420}
\definecolor{xfigc120}{rgb}{0.392,0.392,0.392}
\definecolor{xfigc121}{rgb}{0.718,0.902,1.000}
\definecolor{xfigc122}{rgb}{0.525,0.753,0.925}
\definecolor{xfigc123}{rgb}{0.741,0.741,0.741}
\definecolor{xfigc124}{rgb}{0.827,0.584,0.322}
\definecolor{xfigc125}{rgb}{0.596,0.824,0.996}
\definecolor{xfigc126}{rgb}{0.380,0.380,0.380}
\definecolor{xfigc127}{rgb}{0.682,0.698,0.682}
\definecolor{xfigc128}{rgb}{1.000,0.604,0.000}
\definecolor{xfigc129}{rgb}{0.549,0.612,0.420}
\definecolor{xfigc130}{rgb}{0.969,0.420,0.000}
\definecolor{xfigc131}{rgb}{0.353,0.420,0.224}
\definecolor{xfigc132}{rgb}{0.549,0.612,0.420}
\definecolor{xfigc133}{rgb}{0.549,0.612,0.482}
\definecolor{xfigc134}{rgb}{0.094,0.290,0.094}
\definecolor{xfigc135}{rgb}{0.678,0.678,0.678}
\definecolor{xfigc136}{rgb}{0.969,0.741,0.353}
\definecolor{xfigc137}{rgb}{0.388,0.420,0.612}
\definecolor{xfigc138}{rgb}{0.871,0.000,0.000}
\definecolor{xfigc139}{rgb}{0.678,0.678,0.678}
\definecolor{xfigc140}{rgb}{0.969,0.741,0.353}
\definecolor{xfigc141}{rgb}{0.678,0.678,0.678}
\definecolor{xfigc142}{rgb}{0.969,0.741,0.353}
\definecolor{xfigc143}{rgb}{0.388,0.420,0.612}
\definecolor{xfigc144}{rgb}{0.322,0.420,0.161}
\definecolor{xfigc145}{rgb}{0.580,0.580,0.580}
\definecolor{xfigc146}{rgb}{0.000,0.388,0.000}
\definecolor{xfigc147}{rgb}{0.000,0.388,0.290}
\definecolor{xfigc148}{rgb}{0.482,0.518,0.290}
\definecolor{xfigc149}{rgb}{0.906,0.741,0.482}
\definecolor{xfigc150}{rgb}{0.647,0.710,0.776}
\definecolor{xfigc151}{rgb}{0.420,0.420,0.580}
\definecolor{xfigc152}{rgb}{0.518,0.420,0.420}
\definecolor{xfigc153}{rgb}{0.322,0.612,0.290}
\definecolor{xfigc154}{rgb}{0.839,0.906,0.906}
\definecolor{xfigc155}{rgb}{0.322,0.388,0.388}
\definecolor{xfigc156}{rgb}{0.094,0.420,0.290}
\definecolor{xfigc157}{rgb}{0.612,0.647,0.710}
\definecolor{xfigc158}{rgb}{1.000,0.580,0.000}
\definecolor{xfigc159}{rgb}{1.000,0.580,0.000}
\definecolor{xfigc160}{rgb}{0.000,0.388,0.290}
\definecolor{xfigc161}{rgb}{0.482,0.518,0.290}
\definecolor{xfigc162}{rgb}{0.388,0.451,0.482}
\definecolor{xfigc163}{rgb}{0.906,0.741,0.482}
\definecolor{xfigc164}{rgb}{0.094,0.290,0.094}
\definecolor{xfigc165}{rgb}{0.969,0.741,0.353}
\definecolor{xfigc166}{rgb}{0.000,0.000,0.000}
\definecolor{xfigc167}{rgb}{0.969,0.220,0.161}
\definecolor{xfigc168}{rgb}{0.000,0.000,0.000}
\definecolor{xfigc169}{rgb}{1.000,1.000,0.322}
\definecolor{xfigc170}{rgb}{0.322,0.475,0.290}
\definecolor{xfigc171}{rgb}{0.388,0.604,0.353}
\definecolor{xfigc172}{rgb}{0.776,0.380,0.259}
\definecolor{xfigc173}{rgb}{0.906,0.412,0.259}
\definecolor{xfigc174}{rgb}{1.000,0.475,0.322}
\definecolor{xfigc175}{rgb}{0.871,0.871,0.871}
\definecolor{xfigc176}{rgb}{0.953,0.933,0.827}
\definecolor{xfigc177}{rgb}{0.961,0.682,0.365}
\definecolor{xfigc178}{rgb}{0.584,0.808,0.600}
\definecolor{xfigc179}{rgb}{0.710,0.082,0.490}
\definecolor{xfigc180}{rgb}{0.933,0.933,0.933}
\definecolor{xfigc181}{rgb}{0.518,0.518,0.518}
\definecolor{xfigc182}{rgb}{0.482,0.482,0.482}
\definecolor{xfigc183}{rgb}{0.000,0.353,0.000}
\definecolor{xfigc184}{rgb}{0.906,0.451,0.451}
\definecolor{xfigc185}{rgb}{1.000,0.796,0.192}
\definecolor{xfigc186}{rgb}{0.161,0.475,0.290}
\definecolor{xfigc187}{rgb}{0.871,0.157,0.129}
\definecolor{xfigc188}{rgb}{0.129,0.349,0.776}
\definecolor{xfigc189}{rgb}{0.973,0.973,0.973}
\definecolor{xfigc190}{rgb}{0.902,0.902,0.902}
\definecolor{xfigc191}{rgb}{0.129,0.518,0.353}
\definecolor{xfigc192}{rgb}{0.906,0.906,0.906}
\definecolor{xfigc193}{rgb}{0.443,0.459,0.443}
\definecolor{xfigc194}{rgb}{0.851,0.851,0.851}
\definecolor{xfigc195}{rgb}{0.337,0.620,0.690}
\definecolor{xfigc196}{rgb}{0.788,0.788,0.788}
\definecolor{xfigc197}{rgb}{0.875,0.847,0.875}
\definecolor{xfigc198}{rgb}{0.969,0.953,0.969}
\definecolor{xfigc199}{rgb}{0.800,0.800,0.800}
\clip(888,-4242) rectangle (6492,-201);
\tikzset{inner sep=+0pt, outer sep=+0pt}
\pgfsetlinewidth{+7.5\XFigu}
\pgfsetfillcolor{.!5}
\filldraw (6480,-4230) [rounded corners=+105\XFigu] rectangle (900,-630);
\pgfsetfillcolor{.!10}
\filldraw (4950,-3780) [rounded corners=+105\XFigu] rectangle (1125,-855);
\pgfsetfillcolor{.}
\pgftext[base,left,at=\pgfqpointxy{3690}{-380}] {\fontsize{8}{8.4}\usefont{T1}{ptm}{m}{n}$B_{i-1}$};
\pgftext[base,left,at=\pgfqpointxy{1080}{-4050}] {\fontsize{10}{12}\usefont{T1}{ptm}{m}{n}$G_{i}$};
\pgftext[base,left,at=\pgfqpointxy{1350}{-3600}] {\fontsize{10}{12}\usefont{T1}{ptm}{m}{n}$G_{i-1}$};
\pgftext[base,left,at=\pgfqpointxy{5445}{-380}] {\fontsize{8}{8.4}\usefont{T1}{ptm}{m}{n}$B_{i}$};
\pgftext[base,left,at=\pgfqpointxy{1755}{-2160}] {\fontsize{8}{8.4}\usefont{T1}{ptm}{m}{n}$V_{i-1}\setminus B_{i-1}$};
\pgfsetlinewidth{+30\XFigu}
\draw (3510,-3690) [rounded corners=+105\XFigu] rectangle (4860,-990);
\pgfsetlinewidth{+7.5\XFigu}
\draw (3780,-990)--(3960,-450);
\pgfsetlinewidth{+30\XFigu}
\draw (6390,-2430) [rounded corners=+105\XFigu] rectangle (3600,-1080);
\pgfsetlinewidth{+7.5\XFigu}
\draw (5445,-1080)--(5625,-450);
\pgfsetfillcolor{.!15}
\pgfsetfillcolor{.}
\pgftext[base,left,at=\pgfqpointxy{3780}{-3465}] {\fontsize{8}{8.4}\usefont{T1}{ptm}{m}{n}$B''_{i-1}$};
\pgftext[base,left,at=\pgfqpointxy{5130}{-2115}] {\fontsize{8}{8.4}\usefont{T1}{ptm}{m}{n}$B_i\setminus B_{i-1}$};
\pgftext[base,left,at=\pgfqpointxy{3780}{-2115}] {\fontsize{8}{8.4}\usefont{T1}{ptm}{m}{n}$B'_{i-1}$};
\endtikzpicture%
\caption{$G_i$ and relevant subsets of $V_i$.}\label{fig1}
\end{center} 
\end{figure}
Our approach to determine 
$\pi_i(\sigma_i,\prec_i,\rho_i)$
relies on considering all candidates $\sigma''$
for the restriction of a function $\sigma$
as in the definition of $\pi_i(\sigma_i,\prec_i,\rho_i)$ to the set $\partial V_i\setminus B_i$, that is, $\sigma''=\sigma|_{\partial V_i\setminus B_i}$.

By (iv), 
this restriction satisfies
$\sigma''(\partial V_i\setminus B_i)\leq {t+1\choose 2}-\sigma_i(B_i)$.

In order to exploit the information encoded in the functions 
$\pi_{i-1}(\sigma_{i-1},\prec_{i-1},\rho_{i-1})$, 
we decouple the formation of the hull in $G_i$ within the two graphs $\partial G_i$ and $G_{i-1}$.
Accordingly, 
we consider all candidates for an extension $\prec_{(i-1,i)}$ of $\prec_i$ to $B_i\cup B_{i-1}$ 
specifying a possible order 
in which the vertices in $B_{i-1}\cup B_i$
enter the hull. 
By fixing the pairs $(\sigma'',\prec_{(i-1,i)})$, we specify that $\sigma_{i-1}=\sigma_i|_{B_{i-1}'}+\sigma''|_{B_{i-1}''}$, 
and that the restriction of $\prec_{(i-1,i)}$ to $B_{i-1}$ 
is the linear order on $B_{i-1}$.

Formally, we consider all pairs $(\sigma'',\prec_{(i-1,i)})$, where
\begin{itemize}
\item $\sigma''\in [t]_0^{\partial V_i\setminus B_i}$,
\item $\sigma''(\partial V_i\setminus B_i)\leq 
{t+1\choose 2}-\sigma_i(B_i)$
\item $\prec_{(i-1,i)}$ is a linear extension of $\prec_i$ to $B_i\cup B_{i-1}$.
\end{itemize}
Let $\mathcal{S}$ denote the set of these pairs. 
Note that $|\mathcal{S}|=O\left(n^{t(t+1)/2}(2t-2)!\right)=n^{O(t^2)}$.

Let $(\sigma'',\prec_{(i-1,i)})$ be in $\mathcal{S}$.

Assume that $v_1\prec_{(i-1,i)}\ldots\prec_{(i-1,i)}v_p$ is the linear order on $B_{i-1}\cup B_i$.

For every $j\in[p]$ with $v_j\in B_{i-1}\cup B_i$, 
let $h_j$ be the number of neighbors of $v_j$ in the hull of the set $\{v_\ell: \ell\in[j-1]\}$ in $$\Big(\partial G_i-\big\{v_\ell: \ell\in[p]\setminus[j-1] \big\}, \tau-\sigma_i-\sigma''\Big).$$
If the hull of the set 
$B_{i-1}\cup B_i$ in $(\partial G_i,\tau-\sigma_i-\sigma'')$ 
does not equal 
$\partial V_i$ 
or if $h_j<\tau(v_j)-\sigma_i(v_j)-\rho_i(v_j)$ 
for some $j\in[p]$ with $v_j\in B_i\setminus B_{i-1}$,
then we set all values of $s\big(\sigma'',\prec_{(i-1,i)}\big)$ to $\infty$.
Note that these two cases correspond to violations of the condition (v). 
Hence, in what follows, 
we may assume that these two cases do not hold; 
in particular, $h_j\geq\tau(v_j)-\sigma_i(v_j)-\rho_i(v_j)$ for every $v_j\in B_i\setminus B_{i-1}$. 

Let
\begin{itemize}
\item $\prec_{i-1}$ be the restriction of $\prec_{(i-1,i)}$ to $B_{i-1}$,
\item $\rho_{i-1}(v_j)=\rho_i(v_j)+h_j$ for every $j\in[p]$ with $v_j\in B_{i-1}'$, and
\item $\rho_{i-1}(v_j)=h_j$ for every $j\in[p]$ with $v_j\in B_{i-1}''$.
\end{itemize} 
Let 
$$s\big(\sigma'',\prec_{(i-1,i)}\big)=
\sigma_i|_{B_i\setminus B_{i-1}}
+\sigma''|_{\partial V_i\setminus (B_{i-1}\cup B_i)}
+\pi_{i-1}\big(\sigma_i|_{B{_{i-1}'}}+\sigma''|_{B_{i-1}''},\prec_{i-1}, \rho_{i-1}\big).$$
Note that also in this case some values of $s\big(\sigma'',\prec_{(i-1,i)}\big)$ can be $\infty$.
Note, furthermore, that,
for every $v_j\in B_{i-1}'$,
the value of $\rho_{i-1}(v_j)$
has a contributing term $\rho_i(v_j)$ 
quantifying the help from outside of $V_i$
as well as a contributing term $h_j$ 
quantifying the help from outside of $V_{i-1}$ 
but from inside of $V_i$.
For every $v_j\in B_{i-1}''$,
there is no help from outside of $V_i$, 
that is, the first term disappears.
In view of the above explanation, 
it now follows easily that 
the function $s\big(\sigma'',\prec_{(i-1,i)}\big)$ 
of minimum cost 
with $\big(\sigma'',\prec_{(i-1,i)}\big)$ in ${\mathcal{S}}$ 
is a suitable choice for $\pi_i(\sigma_i,\prec_i,\rho_i)$.
Since $\mathcal{S}$ has $n^{O(t^2)}$ elements, 
and each function $s\big(\sigma'',\prec_{(i-1,i)}\big)$ can be determined in $n^{O(1)}$ time, the claim follows.
\end{proof}

Since $k\leq n$,
and there are only $n^{O(t)}$ local cascades 
for each $G_i$,
the Claims \ref{claim3.1}-\ref{claim3.3} complete the proof of Theorem~\ref{thm:piinterval}.
\end{proof}

\subsection{Approximation algorithms for planar graphs}

First, we use Baker's \cite{ba} layering technique for constructing a PTAS for maximum degenerate sets in planar graphs.

\begin{proof}[Proof of Theorem \ref{thm:ptaskappa}]
Let $\epsilon$ and $(G,\kappa)$ be as in the statement.
We fix a plane embedding of $G$.
Let $L(0),L(1),\ldots,L(\ell)$ be the {\it layers} of $G$, that is, 
\begin{itemize}
\item $L(0)$ is the set of vertices on the outer face of $G$,
\item $L(i)$ is the set of vertices on the outer face of $G-(L(0)\cup\ldots\cup L(i-1))$ for $i\geq 1$,
and 
\item $V(G)=\bigcup\limits_{i=0}^\ell L(i)$. 
\end{itemize}
Let $k= \left\lceil\frac{1}{\epsilon}\right\rceil$.
For every $i\in [k-1]_0$,
the components of the graph $G-X(i)$, 
where $$X(i)=\bigcup\limits_{j=0}^{\left\lfloor\frac{\ell-i}{k}\right\rfloor}L(i+jk),$$
are the subgraphs $G(i,j)$ of $G$ induced by 
$$L\Big(i+(j-1)k+1\Big)\cup L\Big(i+(j-1)k+2\Big)\cup \ldots \cup L\Big(i+jk-1\Big)$$
for $j\in \left[\left\lfloor\frac{\ell-i-1}{k}\right\rfloor+1\right]_0$,
where we set $L(j):=\emptyset$ for $j<0$ or $j>\ell$.
Since each $G(i,j)$ is induced by at most $k$ consecutive layers, the treewidth of these graphs, and, hence, also of $G-X(i)$ is at most $3k-1$ \cite{bo}.
By the main result of Ben-Zwi et al.~\cite{behelone}, and by the duality of dynamic monopolies and degenerate sets,
we can determine in time $n^{O(k)}$
maximum $\kappa\mid_{V(G(i,j))}$-degenerate sets $I(i,j)$ in $G(i,j)$ 
for all $i\in [k-1]_0$ and all $j\in \left[ \left\lfloor\frac{\ell-i-1}{k}\right\rfloor+1\right]_0$.
Let 
$$I(i)=\bigcup\limits_{j=0}^{\left\lfloor\frac{\ell-i-1}{k}\right\rfloor+1}I(i,j)$$
and let
$I$ be the largest of the sets $I(0),\ldots,I(k-1)$.
Clearly, $I$ is a $\kappa$-degenerate set in $G$.

Let $I_{\max}$ be a maximum $\kappa$-degenerate set in $G$.
Since $|I_{\max}|=\sum\limits_{i=0}^{k-1}|I_{\max}\cap X(i)|$,
there is an index $i^*\in [k-1]_0$ with $|I_{\max}\cap X(i^*)|\leq \frac{1}{k}|I_{\max}|\leq \epsilon |I_{\max}|$.
Since $I_{\max}\cap V(G(i,j))$ is a $\kappa\mid_{V(G(i,j))}$-degenerate set in $G_{i,j}$, we obtain
\begin{eqnarray*}
|I|\geq |I(i^*)| &=& \sum\limits_{j=0}^{\left\lfloor\frac{\ell-i^*-1}{k}\right\rfloor+1}|I(i^*,j)|\\
&\geq &\sum\limits_{j=0}^{\left\lfloor\frac{\ell-i^*-1}{k}\right\rfloor+1}|I_{\max}\cap V(G(i^*,j))|\\
&= &|I_{\max}|-|I_{\max}\cap X(i^*)|\\
&\geq& (1-\epsilon)|I_{\max}|,
\end{eqnarray*}
which completes the proof.
\end{proof}

For the proof of Theorem~\ref{thm:apptw}, we adapt the approach of~\cite{aast}.

For a graph $G$ and a subset $A$ of vertices of $G$, let $N^+(A)=\bigcup_{v\in A} N_G(v)\setminus A$ be the {\it outside neighbors} of $A$, and let $N^-(A)=N^+(V(G)\setminus A)$ be the {\it boundary} of $A$.
For a threshold function $\tau$ for $G$ and two sets of vertices $A$ and $B$ of $G$, we call $B$ an {\it $A$-strong region} if $B\not\subseteq H_{(G,\tau)}(A\cup N^+(B))$.
Otherwise we call $B$ an {\it $A$-weak region}.

The following lemma is a straightforward generalization of Lemmas 3.3-3.5 in~\cite{aast}, and can be proved analogously.

\begin{lemma} \label{lem:app}
Let $G$ be a graph with threshold function $\tau$ and let $A$ and $B$ be  sets of vertices of $G$.
\begin{enumerate}[(i)]
\item $B$ is an $A$-strong region if and only if $B\cap(D\setminus A)\neq\emptyset$ for every dynamic monopoly $D$ of $(G,\tau)$ containing $A$.
\item If $B$ is an $A$-weak region with $N^-(B)\subseteq A$, then $B\subseteq H_{(G,\tau)}(A)$.
\item 
If $B$ is an $A$-strong region and $Y$ is a subset of $B$ with $Y\subseteq H_{(G,\tau)}(A)$ and $N^-(Y)\subseteq A$, then $B\setminus Y$ is an $A$-strong region.
\end{enumerate}
\end{lemma}

\begin{proof}[Proof of Theorem~\ref{thm:apptw}]
Let $(G,\tau)$ be as in the statement and let $\left(T,(X_i)_{i\in V(T)}\right)$ be a rooted tree-decomposition of $G$ of width $w$. 
For simplicity, we may assume that $V(T)=[k]$ and that $1,\ldots,k$ is a reversed BFS-ordering of the nodes of $T$, which is rooted at $k$.
As in the definition of a tree-decomposition in subsection~\ref{subsec2.1}, for a node $i$ of $T$, let $G_i$ be the subgraph of $G$ induced by the bags of the node $i$ and all its descendants in $T$.
We describe how to compute a dynamic monopoly for $(G,\tau)$ in time $O(n^2m)$ with approximation guarantee $w+1$. 
Let $A_0=\emptyset,A_1,\ldots,A_k$ be sets of vertices of $G$, recursively defined for every $i\in[k]$, such that
$$A_i=
\begin{cases}
A_{i-1}\cup X_i& \mbox{, if $V(G_i)$ is an $A_{i-1}$-strong region, and}\\
A_{i-1} & \mbox{, if $V(G_i)$ is an $A_{i-1}$-weak region.}
\end{cases}$$

By the definition of a tree-decomposition, the boundary vertices of $G_i$ lie in the bag $X_i$, that is, $N^-(V(G_i))\subseteq X_i$, and for every child node $j$ of $i$, the boundary vertices of $G_{j}$ lie in $X_i\cap X_{j}$, that is, $N^-(V(G_{j}))\subseteq X_i\cap X_{j}$, and thus $N^-(V(G_{j}))\subseteq X_i$.

First, we show that the set $A_k$ is a dynamic monopoly for $(G,\tau)$.
\begin{claim*}
Let $i\in[k]$. If $V(G_i)$ is an $A_{i-1}$-strong region, then $V(G_i)\subseteq H_{(G,\tau)}(A_{i-1}\cup X_i)$.
\end{claim*}
\begin{proof}[Proof of the Claim]
We prove the claim by induction on the height of the subtree $T_i$.
The statement is true if $i$ is a leaf of $T$ since $V(G_i)=X_i$.
Hence, let $i_1,\ldots,i_{\ell}$ be the children of $i$ in $T$.
For every $j\in[\ell]$, when the algorithm examined $V(G_{i_j})$ either $V(G_{i_j})$ was $A_{i_j-1}$-weak or $A_{i_j-1}$-strong. 
\\ \noindent If $V(G_{i_j})$ was $A_{i_j-1}$-weak, by Lemma~\ref{lem:app} (ii), we have 
$$V(G_{i_j})\subseteq H_{(G,\tau)}\Big(A_{i_j-1}\cup N^-\left(V(G_{i_j})\right) \Big)
\subseteq H_{(G,\tau)}\Big(A_{i_j-1}\cup X_i \Big)
\subseteq H_{(G,\tau)}\Big(A_{i}\cup X_{i} \Big).$$
If $V(G_{i_j})$ was $A_{i_j-1}$-strong, by induction, we have 
$$V(G_{i_j})\subseteq H_{(G,\tau)}\left(A_{i_j-1}\cup X_{i_j}\right)\subseteq H_{(G,\tau)}\left(A_i\right).$$

Altogether, since $V(G_i)=X_i\cup\bigcup\limits_{j=1}^{\ell} V(G_{i_j})$, this proves the claim.
\end{proof}
We obtain that $A_k$ is a dynamic monopoly for $(G,\tau)$, because $V(G_k)=V(G)$ and
\begin{itemize}
\item either 
$V(G_k)$ is an $A_{k-1}$-strong region, thus, $A_k=A_{k-1}\cup X_k$ and by the above claim $V(G_k)\subseteq H_{(G,\tau)}(A_{k-1}\cup X_k)$, 
\item or $V(G_k)$ is an $A_{k-1}$-weak region, thus, $A_k=A_{k-1}$ and again by Lemma~\ref{lem:app}~(ii), $V(G_k)\subseteq H_{(G,\tau)}\big(A_{k-1}\cup N^-\left(V(G_k)\right)\big)=H_{(G,\tau)}(A_k)$ since $N^-(V(G_k))=\emptyset$.
\end{itemize}

\medskip
Let $i_1\leq i_2 \leq \ldots \leq i_{b}$ be all the nodes of $T$ that each corresponds to an $A_{i_j-1}$-strong region.
For every $j\in[b]$, let $B_j=V(G_{i_j})\setminus\left(\bigcup\limits_{\ell=1}^{j-1}V(G_{i_\ell})\right)$.
We claim that $B_j$ is an $A_{i_j-1}$-strong region for every $j\in[b]$.
For every $\ell\in[j-1]$, we have $N^-\left(V(G_{i_{\ell}})\right)\subseteq X_{i_{\ell}}\subseteq A_{i_j}$ and 
$V(G_{i_{\ell}})\subseteq H_{(G,\tau)}(A_{i_j})$.
This implies that 
$N^-\left(\bigcup\limits_{\ell=1}^{j-1} V(G_{i_{\ell}})\right)\subseteq A_{i_j}$ and 
$\bigcup\limits_{\ell=1}^{j-1} V(G_{i_{\ell}})\subseteq H_{(G,\tau)}(A_{i_j})$.
Hence, by Lemma~\ref{lem:app}~(iii), $B_j$ is $A_{i_j-1}$-strong.
By construction, the sets $B_1,\ldots,B_b$ are pairwise disjoint, and for each such set, a bag $X_{i_j}$ is added to $A_{i_j-1}$.
This implies that 
$$A_k\leq (w+1)b\leq (w+1){\rm dyn}(G,\tau),$$
since each bag has size at most $w+1$ and ${\rm dyn}(G,\tau)\geq b$, by Lemma~\ref{lem:app}~(i). 
This proves the approximation guarantee.

Finally, we show that the described algorithm can be implemented with running time $O(n^2m)$.
Recall that the given tree-decomposition has $O(n)$ nodes.
For each node $i$, it can be checked whether $V(G_i)$ is an $A_{i-1}$-strong region by constructing the hull $H_{(G,\tau)}\Big(A_{i-1}\cup N^+\big(V(G_i)\big)\Big)$ in time $O(nm)$.
This completes the proof.
\end{proof}

Corollary~\ref{cor:appplanar} follows easily from Theorem~\ref{thm:apptw} and the fact that for a planar graph, a tree-decomposition of width $O(\sqrt{n})$ can be found in time $O(n^{3/2})$~\cite{albofeklni}.

In fact, in~\cite{dvno} it was shown that if every subgraph of a graph $G$ with order $n$ has a balanced separator of size at most $k$ then $G$ has treewidth at most $105k$, where a balanced separator $S$ is a set of vertices of $G$ such that every component of $G-S$ has size at most $2n/3$.
The balanced separation number is the smallest integer $k$ fulfilling this condition.
Hence, for every graph of order $n$ and balanced separation number at most $O(n^{1-\epsilon})$, Theorem~\ref{thm:apptw} yields an approximation algorithm with ratio $O(n^{1-\epsilon})$, given the corresponding tree-decomposition.
Note that the result of Lipton and Tarjan~\cite{lita} immediately implies a balanced separation number of size $O(\sqrt{n})$ for planar graphs.


\begin{thebibliography}{99}
\bibitem{aast} A. Aazami, K. Stilp, Approximation algorithms and hardness for domination with propagation, SIAM Journal on Discrete Mathematics 23.3 (2009) 1382-1399.
\bibitem{acbewo} E. Ackerman, O. Ben-Zwi, G. Wolfovitz, Combinatorial model and bounds for target set selection, Theoretical Computer Science 411 (2010) 4017-4022.
\bibitem{albofeklni} J. Alber, H.L. Bodlaender, H. Fernau, T. Kloks, R. Niedermeier, Fixed parameter algorithms for dominating set and related problems on planar graphs, Algorithmica 33 (2002) 461-493.
\bibitem{arcopr} S. Arnborg, D.G. Corneil, A. Proskurowski, Complexity of finding embeddings in a $k$-tree, SIAM Journal on Algebraic and Discrete Methods 8 (1987) 277-284.
\bibitem{ba} B.S. Baker, Approximation algorithms for NP-complete problems on planar graphs, Journal of the ACM 41 (1994) 153-180.
\bibitem{barasasz} R. Barbosa, D. Rautenbach, V. Fernandes dos Santos, J.L. Szwarcfiter, On minimal and minimum hull sets, Electronic Notes in Discrete Mathematics 44 (2013) 207-212.
\bibitem{behelone} O. Ben-Zwi, D. Hermelin, D. Lokshtanov, I. Newman, Treewidth governs the complexity of target set selection, Discrete Optimization 8 (2011) 87-96.
\bibitem{beehpera} S. Bessy, S. Ehard, L.D. Penso, D. Rautenbach, Dynamic monopolies for interval graphs with bounded thresholds, arXiv:1802.03935.
\bibitem{bo} H.L. Bodlaender, A partial k-arboretum of graphs with bounded treewidth, Theoretical Computer Science 209 (1998) 1-45.
\bibitem{bolu} K.S. Booth, G.S. Lueker, Testing for the consecutive ones property, interval graphs, and graph planarity using PQ-tree algorithms, Journal of Computer and System Sciences 13 (1976) 335-379.
\bibitem{cedoperasz} C.C. Centeno, M.C. Dourado, L.D. Penso, D. Rautenbach, J.L. Szwarcfiter, Irreversible conversion of graphs, Theoretical Computer Science 412 (2011) 3693-3700.
\bibitem{ch} N. Chen, On the approximability of influence in social networks, SIAM Journal on Discrete Mathematics 23 (2009) 1400-1415.
\bibitem{chhuliwuye} C.-Y. Chiang, L.-H. Huang, B.-J. Li. J. Wu, H.-G. Yeh, Some results on the target set selection problem, Journal of Combinatorial Optimization 25 (2013) 702-715.
\bibitem{cogareva} G. Cordasco, L. Gargano, A. Rescigno, U. Vaccaro, Optimizing spread of influence in social networks via partial incentives, Lecture Notes in Computer Science 9439 (2015) 119-134. 
\bibitem{deha} E.D. Demaine and M. Hajiaghayi, The bidimensionality theory and its algorithmic applications, Computer Journal 51 (2008) 292-302.
\bibitem{drro} P.A. Dreyer Jr., F.S. Roberts, Irreversible k-threshold processes: Graph-theoretical threshold models of the spread of disease and of opinion, Discrete Applied Mathematics 157 (2009) 1615-1627.
\bibitem{dvno} Z. Dvor\'ak, S. Norin, Treewidth of graphs with balanced separations, arXiv:1408.3869.
\bibitem{gakost} M.R. Garey, D.S. Johnson, L. Stockmeyer, Some simplified NP-complete graph problems, Theoretical Computer Science 1 (1976) 237-267.
\bibitem{keklta} D. Kempe, J. Kleinberg, E. Tardos, Maximizing the spread of influence through a social network, Theory of Computing 11 (2015) 105-147.
\bibitem{kl} T. Kloks, Treewidth. Computations and approximations, Lecture Notes in Computer Science 842 (1994).
\bibitem{kylivy} J. Kyn\v{c}l, B. Lidick\'{y}, T. Vysko\v{c}il, Irreversible $2$-conversion set in graphs of bounded degree, Discrete Mathematics and Theoretical Computer Science 19 (2017) \# 3.
\bibitem{lita}R.J. Lipton, R.E. Tarjan, A separator theorem for planar graphs, SIAM Journal on Applied Mathematics 36 (1979) 177-189.
\bibitem{re} D. Reichman, New bounds for contagious sets, Discrete Mathematics 312 (2012) 1812-1814.
\end{thebibliography}
\end{document}